\documentclass [11pt] {article}
\usepackage{amsfonts}
\usepackage{amssymb}
\usepackage{times}
\usepackage{graphicx}

\newtheorem{lemma}{Lemma}[section]
\newtheorem{corollary}{Corollary}[section]
\newtheorem{theorem}{Theorem}

\newtheorem{proposition}{Proposition}[section]
\newtheorem{definition}{Definition}[section]

\def\tr{\mbox{Tr}\,}

\def\limo{\lim_{\omega}}

\def\dim{{\rm dim}}

\def\<{\langle}
\def\>{\rangle}

\def\bX{{\bf X}}
\def\bY{{\bf Y}}
\def\bA{{\bf A}}
\def\bB{{\bf B}}
\def\bh{{\bf h}}

\def\bm{{\bf \mu}}

\begin{document}

\title{Higher order fourier analysis as an algebraic theory II.}
\author{{\sc Bal\'azs Szegedy}}

\maketitle

\abstract{Our approach to higher order Fourier analysis is to study the ultra product of finite (or compact) Abelian groups on which a new algebraic theory appears. This theory has consequences on finite (or compact) groups usually in the form of approximative statements.
The present paper is a continuation of \cite{Sz1} in which higher order characters and decompositions were introduced. We generalize the concept of the Pontrjagin dual group and introduce higher order versions of it. We study the algebraic structure of the higher order dual groups. We prove a simple formula for the Gowers uniformity norms in terms of higher order decompositions.
We present a simple spectral algorithm to produce higher order decompositions. We briefly study a multi linear version of Fourier analysis. Along these lines we obtain new inverse theorems for Gowers's norms.
}

\tableofcontents

\section{Introduction}

Higher order Fourier analysis was started by Gowers \cite{Gow},\cite{Gow2} generalizing Roth's Fourier analytic proof for Roth's theorem.
In this framework Gowers introduced a sequence of norms $\|f\|_{U_k}$ for functions on finite (or compact) Abelian groups. Let $A$ be a compact Abelian group with the normalized Haar measure.
For an $L_\infty$ measurable function $f:A\rightarrow\mathbb{C}$ and $t\in A$ we introduce the function $\Delta_t f$ whose value at $x$ is $f(x+t)\overline{f(x)}$. By iterating this we also introduce $\Delta_{t_1,t_2,\dots,t_k}f$. The norm $\|f\|_{U_k}$ is the $2^k$-th root of the expected value of $\Delta_{t_1,t_2,\dots,t_k}f(x)$ when $t_1,t_2,\dots,t_k$ and $x$ are chosen uniformly at random.

Quite interestingly, it turns out that if $f$ is bounded then the norm $\|f\|_{U_2}$ decides if $f$ is correlated with a harmonic function or $f$ is just "random noise" (from the Fourier theoretic point of view).
To be more precise, it is not hard to see that
\begin{equation}\label{gow2}
\|f\|_{U_2}=\Bigl(\sum |\lambda_i|^4\Bigr)^{1/4}
\end{equation}
where the numbers $\lambda_i$ are the Fourier coefficients of $f$.
It follows that if $\|f\|_2\leq 1$ then $$|\lambda_{\max}|\leq\|f\|_{U_2}\leq|\lambda_{\max}|^{1/2}$$ for the maximal Fourier coefficient $\lambda_{\max}$.

\smallskip

It is possible that $\|f\|_{U_2}$ is arbitrarily small but $\|f\|_{U_3}$ is a fixed positive constant whereas for the ``completely random noise'' all the uniformity norms are very small. This suggest that functions that are noise from the Fourier theoretic point of view can have structure measured by the higher uniformity norms.
The natural question arises:

\smallskip

{\center{\bf What do the norms $U_k$ measure in general?}}

\smallskip

One can reasonably expect that either $U_k$ measures something in Fourier analysis which is more complicated than the dominant coefficient or there is a generalization of Fourier analysis (a higher order version) that is relevant to the meaning of $U_k$.
Even though the possibility of a pure Fourier theoretic approach can't be excluded, quite amazingly, it turns out that reality favors the existence of a higher order version of Fourier analysis.
A work by Green and Tao \cite{GrTao} shows for example that if $\|f\|_{U_3}$ is separated from $0$ then $f$ is correlated with a ``quadratic structure''.

Another interesting fact is that the Higher order version of Fourier analysis, in a strict algebraic form, seems to appear only in the limit.
We have to mention that a lot of work has been done in the finite case by Gowers, Green, Tao, and many others. Also limiting theories were studied by Tao, Ziegler, Host, Kra and others.

In this sequence of papers we follow a different approach than previous papers. We build up a theory that is strongly related to hypergraph regularization. For this reason we use methods from the paper \cite{ESz} in which the limits of hypergraph sequences are studied.
Higher order Fourier analysis from our point of view is a theory that, in a precise algebraic form, appears on the ultra product of compact Abelian groups. Even though such groups are not compact (they satisfy a weakening of compactness), results in higher order Fourier analysis can be used to prove statements about measurable functions on compact Abelian groups.

To keep things simple, most of this paper deals with the case when the groups (of which we take the ultra product) are finite, however the technique that we use is not limited to this case. As we will point out, most results can be extended to compact, (infinite) discrete, and some results to locally compact groups.

An important feature of our theory is that while on a compact abelian group $A$ linear characters span the whole Hilbert space $L_2(A)$ it is no longer true on an ultra product group $\bA$. This creates space for an extended theory.
Linear characters of $\bA$ span a subspace in $L_2(\bA)$ which is the $L_2$ of an $\bA$-invariant $\sigma$-algebra. We denote this $\sigma$-algebra by $\mathcal{F}_1$ and we call it the first Fourier $\sigma$-algebra. Note that a finite analogy of an $\mathcal{F}_1$ measurable set is a set that is composed of a bounded number of Bohr neighborhoods. It turns out \cite{Sz1} that on $\bA$ an interesting hierarchy of $\sigma$-algebras $\mathcal{F}_0\subset\mathcal{F}_1\subset\mathcal{F}_2\dots$ appears (where $\mathcal{F}_0$ is the trivial $\sigma$-algebra). In general $L_2(\mathcal{F}_k)$ is the orthogonal space of the space of functions with $\|f\|_{U_{k+1}}=0$. Note that $U_k$ for the group $\bA$ is a norm only on $L_\infty(\mathcal{F}_{k-1})$ and is a semi-norm in general.

The fundamental theorem in our theory (proved in \cite{Sz1}) says that $L_2(\mathcal{F}_k)$ can be uniquely decomposed into the orthogonal sum of the ($L_2$-closures of) rank one modules over the algebra $L_\infty(\mathcal{F}_{k-1})$. In particular $L_2(\mathcal{F}_1)$ is decomposed into rank one modules over $L_\infty(\mathcal{F}_0)=\mathbb{C}$ which is ordinary Fourier theory. Our fundamental theorem immediately gives rise to a so-called $k$-th order Fourier decomposition. For every fixed $k$ each bounded measurable function $f$ on $\bA$ can be uniquely decomposed as $$f=g+f_1+f_2+f_3+\dots$$ converging in $L_2$ such that $\|g\|_{U_{k+1}}=0$ and $f_i$ are chosen from distinct rank one modules over $L_\infty(\mathcal{F}_{k-1})$. The components $f_i$ are the projections of $f$ to the different rank one modules.

Now we list the topics investigated in this paper:

\bigskip

\noindent{\bf Additivity result for the Gowers norms:}~We show that the $2^{k+1}$-th power of the norm $U_{k+1}$ is additive on the terms of the $k$-th order decomposition of $f$. This generalizes the equation (\ref{gow2}).

\bigskip

\noindent{\bf Local behavior of higher order characters:}~We show that the $k$-th order components $f_i$ have a ``locally $k$-th degree'' structure relative to a separable $\sigma$-algebra contained in $\mathcal{F}_{k-1}$ (see lemma \ref{charsep} and corollary \ref{secder}). Intuitively it means that the $k$-th order components $f_i$ have a $k$-th degree behavior on neighborhoods that have a degree $k-1$ structure. Furthermore the whole group $\bA$ is covered by these neighborhoods corresponding to $f_i$. It is important to mention that here by a degree $k$ function we mean a function whose $k$-th difference function is constant.

\bigskip

\noindent{\bf Spectral approach:}~For every $k$ and $f$ we construct an operator in a simple way whose eigenvectors are the $k$-th order components $f_1,f_2,\dots$ of $f$. We call this fact the ``spectral form of the fundamental theorem''. (An interesting feature is that a $k$-th order generalization of convolution seems to appear here.)
This part of the paper gives some hope that a finite version of $k$-th order decompositions could be efficiently computed. At the end of the paper we sketch a finite translation of this topic but details will be worked out in a separate paper.

\bigskip

\noindent{\bf Limits of self adjoint operators:}~The spectral approach can be best translated to finite results if a limit theory for operators is available. We prove a theorem in this direction which has consequences of independent interest. For example it implies a spectral version of the strong regularity lemma. It also implies that for a convergent graph sequence the joint distributions of the (entries of the) eigenvectors corresponding to the largest eigenvalues converge. (Note that it is easy to see that the eigenvalues converge) Further details on this will be available in a separate paper \cite{Sz2}.

\bigskip

\noindent{\bf Higher order dual groups:}~Rank one modules over $L_\infty(\mathcal{F}_{k-1})$ are forming an Abelian group that we call the $k$-th order dual group and denote by $\hat{\bA}_k$. We study the algebraic structure of $\hat{\bA}_k$ which turns out to be closely related to the dual of the $k$-th tensor power of $\bA$. This also justifies the name ``Higher order Fourier Analysis''. To be more precise we show that every element in $\hat{\bA}_k$ can be characterized by an element of
$\hom(\bA,\hat{\bA}_{k-1}/T)$
where $T$ is a countable subgroup. Such countable subgroups will be called ``types'' in general. Note that in some sense $T$ encodes the dual concept of a $k-1$ degree ``neighborhood''. In particular, if $\phi$ is a ``pure'' $k$-th degree character, i.e. $\Delta_{t_1,t_2,\dots,t_{k+1}}\phi(x)$ is the constant $1$ function, then the element of $\hat{\bA}_k$ representing $\phi$ has the trivial group as $T$. This corresponds to the fact that in this case $\phi$ is nice on the whole group. Countable subgroups $T$ of the first dual group $\hat{\bA}_1$ define Bohr neighborhoods in our language. The intersection $K$ of the kernels of the characters in $T$ is a subgroup of $\bA$ which is the ultra product version of a Bohr neighborhood. The factor group $\bA/T$ has a compact topological group structure.

The homomorphism sets $\hom(\bA,\hat{\bA}_{k-1}/T)$ as $T$ runs trough all countable subgroups are forming a direct system. We denote the direct limit by $\hom^*(\bA,\hat{\bA}_{k-1})$. We prove that $\hat{\bA}_k$ is injectively embedded into $\hom^*(\bA,\hat{\bA}_{k-1})$. This has numerous consequences on the structures of the dual groups. For example it shows that $\hat{\bA}_k$ is embedded into
$$\hom^*(\bA,\hom^*(\bA,\dots,\hom^*(\bA,\hom(\bA,\mathbb{C})))\dots )$$
where the number of $\hom^*$'s is $k-1$. This shows the connection to the $k$-th tensor power of $\bA$.

\bigskip

\noindent{\bf Multi linear Fourier analysis:}~In a short chapter, out of aesthetical reasons, we put down the foundations of ``multi linear Fourier analysis'' which seems to simplify Higher order Fourier analysis. The main point is that degree $k$ functions (functions measurable in $\mathcal{F}_k$) can be represented (without loosing too much information) in a $\sigma$-algebra whose elements are in some sense $k$-linear. The representation is given by the simple formula
$$\tilde{V}_k(f)(t_1,t_2,\dots,t_k)=\int_x\Delta_{t_1,t_2,\dots,t_k}f(x).$$
This brings Higher order Fourier analysis one step closer to usual (linear) Fourier analysis. A good example is that if $\phi$ is a pure $k$-th order character then $\Delta_{t_1,t_2,\dots,t_k}f(x)$ doesn't depend on $x$ and is $k$-linear in $t_1,t_2,\dots,t_k$. Degree $k$-pure characters in characteristic $p$ can look difficult but the representing multi linear form takes only $p$-th root of unities.

\bigskip

\noindent{\bf Finite translations and the inverse theorem:}~Last but not least we translate some of our results into finite statements to demonstrate how it works.
In particular we get an inverse theorem for the uniformity norms (see Theorem \ref{furreg}). Roughly speaking it says that if $\|f\|_{U_k}$ is separated from $0$ then $f$ is correlated with a function $f'$ that has a degree $k-1$ structure structure. This means that the abelian group can be decomposed into a bounded number of sets that are level sets of functions with a degree $k-2$ structure such that the value of $\Delta_{t_1,t_2,\dots,t_k}f'(x)$ depends only (with a high probability and with a small error) on the partition sets containing the partial sums $x+\sum_{i\in S} t_i$. Notice that the notion of a ``function with a degree $k$ structure'' is recursively defined. As a starting step, functions with a zero degree structure are only the constant functions. This means that functions with a first degree structure are close to ordinary characters. It is important to note that the theorem is close to the so called strong regularity lemma in spirit. For example the error in the definition of a degree $k$ structure can be made arbitrary small in terms of the number of partition sets.

\bigskip

\noindent{\bf Remarks:}~Some important topics are not discussed in this paper including connections to nilpotent groups, and the non-commutative theory. They will be discussed in a subsequent paper.
A third part of this sequence is in preparation.

\subsection{Overview of the first paper}

We denote by $\bA$ the ultra product of the sequence $\{A_i\}_{i=1}^\infty$ of growing finite Abelian groups.
The set $\bA$ is a probability space with the group invariant $\sigma$-algebra $\mathcal{A}$ and measure $\bm$ introduced in \cite{Sz1}.

Let $\mathcal{F}_0$ denote the trivial $\sigma$-algebra on $\bA$. A measurable set $S\in\mathcal{A}$ is called separable if the $\sigma$-algebra generated by all translates $S+x$ is separable i.e. it can be generated by countable many sets.
If $\mathcal{B}$ is a shift invariant $\sigma$-algebra in $\mathcal{A}$ and $S\in\mathcal{A}$ then we say that $S$ is relative separable over $\mathcal{B}$ if the sigma algebra generated by $\mathcal{B}$ and all translates $S+x$ can be generated by $\mathcal{B}$ and countably many extra sets. Separable elements in $\mathcal{A}$ are forming a $\sigma$-algebra that we denote by$\mathcal{F}_1$. We construct a growing sequence of $\sigma$-algebras $\mathcal{F}_i$ recursively such that $\mathcal{F}_i$ is formed by all measurable sets which are relative separable over $\mathcal{F}_{i-1}$.

Let $\mathcal{C}$ denote the set $\{x|x\in\mathbb{C},|x|=1\}$.
For a function $f:\bA\rightarrow\mathbb{C}$ and $t\in\bA$ we introduce the notation $$\Delta_t f(x)=f(x+t)\overline{f(x)}.$$ Similarly we define $\Delta_{t_1,t_2,\dots,t_k}f(x)$ by iterating the previous operation. The $2^k$-th root of the expected value of $\Delta_{t_1,t_2,\dots,t_k}f(x)$ where $x,t_1,t_2,\dots,t_k$ are chosen uniformly at random is called the $k$-th Gowers (semi)-norm of $f$.

It was proved in \cite{Sz1} that

\begin{lemma}[Norm characterization]\label{normchar} For a measurable function $f\in L_\infty(\mathcal{A},\bm)$ we have that
$$\|f\|_{U_k}=\|\mathbb{E}(f|\mathcal{F}_{k-1})\|_{U_k}.$$
Furthermore the $k$-th Gowers (semi)-norm of $f$ is $0$ if and only if $f$ is orthogonal to $L_2(\mathcal{F}_{k-1},\bm)$. In particular $U_k$ is a norm on $L_\infty(\mathcal{F}_{k-1})$.
\end{lemma}

We denote the $k$-th Gowers norm either by $U_k(f)$ or by $\|f\|_{U_k}$.

The fundamental theorem proved in \cite{Sz1} says that

\begin{theorem}[Fundamental theorem]\label{fundth} $L_2(\mathcal{F}_k,\bm)$ is equal to the orthogonal sum of all the ($L_2$-closures of) shift invariant rank one modules over the algebra $L_\infty(\mathcal{F}_{k-1},\bm)$. Every (shift invariant) rank one module over $L_\infty(\mathcal{F}_{k-1},\bm)$ has a generating element $\phi:\bA\rightarrow\mathcal{C}$ such that the functions
$\Delta_t\phi$
are measurable in $\mathcal{F}_{k-1}$ for every $t\in\bA$.
Such functions are called $k$-th order characters.
\end{theorem}

It is clear that $k$-th order characters are unique up to multiplication with elements from $L_\infty(\mathcal{F}_{k-1},\bm)$ of absolute value $1$.

Every function $f\in L_2(\mathcal{A},\bm)$ has a unique (possibly infinite) decomposition
$$f=f_0+f_1+f_2+\dots$$
where $f_0$ is orthogonal to $\mathcal{F}_k$ and for $i>0$ the functions $f_i$ are contained in distinct rank $1$ modules over $L_\infty(\mathcal{F}_{k-1},\bm)$.

\subsection{Compact groups, discrete groups and other generalizations}

In our paper usually $\bA$ is the ultra product of a sequence $\{A_i\}_{i=1}^\infty$ of finite abelian groups. However finiteness is not significant here. One can replace $A_i$ by compact groups with the normalized Haar measure on them. In this case the $\sigma$-algebra on the ultra product is generated by ultra products of Borel sets and the measure is the ultra limit of the Haar measures.

More interestingly one can extend the theory to the case where $A_i$ are discrete infinite Abelian groups with finitely additive measures $\tau_i$ with total measure one. Since Abelian groups are amenable, such measures always exist. In this case, quite interestingly, the ultra product becomes an ordinary measure space. Again the $\sigma$-algebra is generated by ultra products of subsets of $A_i$.

Some of our theory can be generalized to locally compact groups, but one has to be much more careful in that case. We will discuss that in a separate paper.

Finally, an interesting direction is the case when the groups $A_i$ are not commutative.
In this situation a ``higher order representation theory can be studied''
This is also a topic of a forthcoming paper.

\section{Basics}

\subsection{Shift invariant $\sigma$-algebras and $\mathcal{B}$-orthogonality}

\begin{lemma} Let $\mathcal{B}$ be a shift invariant $\sigma$-algebra in $\mathcal{A}$. If $\phi:\bA\rightarrow\mathcal{C}$ satisfy
that $\Delta_t\phi$ is in $L_\infty(\mathcal{B},\bm)$ for almost every $t$ then it is true for every $t$.
\end{lemma}

\begin{proof}Let $t$ be arbitrary and $t_2$ be chosen randomly. Then with probability one both
$\phi(x+t)\overline{\phi(x+t_2)}$ and $\phi(x+t_2)\overline{\phi(x)}$ are measurable in $\mathcal{B}$ with probability one. Taking such a $t_2$ and the product of the two expressions the proof is complete.
\end{proof}

\begin{lemma}\label{ort1} Let $\mathcal{B}$ be a shift invariant $\sigma$-algebra. Let $\phi:\bA\rightarrow\mathcal{C}$ be a function such that $\Delta_t\phi$ is measurable in $\mathcal{B}$ for every $t\in\bA$. Then either $\mathbb{E}(\phi|\mathcal{B})$ is constant $0$ or $\phi\in\mathcal{B}$.
\end{lemma}

\begin{proof} Let $\phi^*=\mathbb{E}(\phi|\mathcal{B})$. Then for every $t$ we have that

$$\phi^*(x+t)=\mathbb{E}(\phi(x+t)|\mathcal{B})=\mathbb{E}(\phi(x)(\phi(x+t)\overline{\phi(x)})|\mathcal{B})=\phi^*(x)\phi(x+t)\overline{\phi(x)}$$
for almost all $x$.
Assume that $\phi^*$ is not $0$ on a positive measure set. We obtain that there is a fixed $x$ with $\phi^*(x)\neq 0$ such that
$$\phi^*(x+t)=\phi^*(x)\phi(x+t)\overline{\phi(x)}$$
for almost all $t$.
Since $|\phi(x)|=1$, this implies that $\phi(x+t)$ is measurable in $\mathcal{B}$ and its shifted version $\phi(t)$ is also measurable in $\mathcal{B}$.
\end{proof}

\begin{lemma}\label{sepex} Let $\mathcal{B}\subseteq\mathcal{A}$ be a shift invariant $\sigma$-algebra and $\phi:\bA\rightarrow\mathcal{C}$ be a function such that $\Delta_t\phi\in\mathcal{B}$ for every $t\in\bA$. Then $\phi$ is contained in a shift invariant separable extension of $\mathcal{B}$.
\end{lemma}

\begin{proof} Since the level sets of $\phi$ can be generated countable many sets it is enough to prove that the $\sigma$-algebra $\mathcal{B}'$ generated by $\phi$ and $\mathcal{B}$ is again shift invariant.
The formula $\phi(x+t)=\phi(x)/\Delta_t\phi$ implies that for every fixed $t$ the function $x\mapsto\phi(x+t)$ is measurable in $\mathcal{B}'$. It follows that any shift of $\mathcal{B}'$ is contained in $\mathcal{B}'$. Since shifts have inverses the proof is complete.
\end{proof}

\begin{definition} Let $\mathcal{B}\subseteq\mathcal{A}$ be a $\sigma$-algebra. Two measurable functions $f,g$ on $\bA$ are called $\mathcal{B}$-orthogonal if $\mathbb{E}(f\overline{g}|\mathcal{B})$ is the constant $0$ function.
\end{definition}

\begin{lemma}\label{ort2} Let $f,g$ be elements from two distinct rank one module over $L_\infty(\mathcal{F}_{k-1})$. Then $f$ and $g$ are $\mathcal{F}_{k-1}$ orthogonal.
\end{lemma}

\begin{proof} We can assume that $f$ and $g$ are $k$-th order characters. It is clear that $h=fg$ is a $k$-th order character and that it is not measurable in $\mathcal{F}_{k-1}$. Now lemma \ref{ort1} completes the proof.
\end{proof}

The next lemma is an immediate consequence of lemma \ref{normchar}

\begin{lemma}\label{gort} Two functions $f,g$ are $\mathcal{F}_{k-1}$ orthogonal if and only if $\|f\overline{g}\|_{U_k}=0$.
\end{lemma}

\subsection{The fundamental theorem}

In this section we elaborate on the fundamental theorem. We restate it in a more detailed form.

\begin{theorem}[Variant of the Fundamental theorem]\label{dec1} Fix a natural number $k$.
\begin{enumerate}
\item Every function $f\in L_\infty(\bA,\mathcal{A})$ can be uniquely decomposed into two parts $f=f'+g$ where $\|f'\|_{U_{k+1}}=\|f\|_{U_{k+1}}$ and $\|g\|_{U_{k+1}}=0$ such that $f'$ is orthogonal to any function $h$ with $\|h\|_{U_{k+1}}=0$ (or equivalently $f'$ is measurable in $\mathcal{F}_k$). The function $f'$ is equal to the projection $\mathbb{E}(f|\mathcal{F}_k)$.

\item The function $f'$ can be uniquely decomposed into the $L_2$ sum of functions $\{f_i\}_{i=1}^\infty$ with $\|f'\|_p\geq\|f_i\|_p$ for all $i\in\mathbb{N},~1\leq p\leq \infty$ such that $\|f_i\overline{f_j}\|_{U_k}=0$ for $i\neq j$ and $\Delta_t f_i$ is orthogonal to any function $h$ with $\|h\|_{U_k}=0$ (or equivalently $\Delta_t f_i\in\mathcal{F}_{k-1}$).

\item For every $f_i$ there is a $k$-th order character $\phi_i:\bA\rightarrow\mathcal{C}$ and function $g_i$ measurable in $\mathcal{F}_{k-1}$ such that $f_i=\phi_ig_i$. In particular $|f_i|=g_i\in\mathcal{F}_{k-1}$.
\end{enumerate}
\end{theorem}

\begin{proof} Most of the parts are trivially equivalent with the original form.
We prove the inequality $\|f'\|_p\geq\|f_i\|_p$ for $p\geq 1$. Let $\phi$ be a character representing the module of $f_i$. Then by lemma \ref{ort2} we have that $f_i=\mathbb{E}(f'\overline{\phi}|\mathcal{F}_{k-1})\phi$. It follows that $$\|f_i\|_p=\|\mathbb{E}(f'\overline{\phi}|\mathcal{F}_{k-1})\|_p\leq\|f'\overline{\phi}\|_p=\|f'\|_p.$$

The equality $\|f_i\overline{f_j}\|_{U_k}=0$ in part two also follows from lemma \ref{ort2}.
\end{proof}

\subsection{Higher order decomposition as spectral decomposition}

Let $D_k$ denote the subspace of all functions $f$ in $L_2(\bA\times\bA,\mathcal{F}_k\times\mathcal{F}_k)$ with the property that the functions $x\mapsto f(x,x+t)$ is in $\mathcal{F}_{k-1}$ for (almost) all $t\in\bA$.

\begin{definition} Let $k$ be a fixed natural number. We define the operator $$\mathcal{K}_k:L_\infty(\bA,\mathcal{A})\rightarrow L_\infty(\bA\times\bA,\mathcal{F}_k\times\mathcal{F}_k)$$ such that $\mathcal{K}_k(f)$ is the orthogonal projection of $h(x,y):=f(x)\overline{f(y)}$ to $D_k$
\end{definition}

From this definition it is not even clear that $\mathcal{K}_k(f)$ is in $L_\infty$.
This follows from the next lemma.

\begin{lemma} Let $k$ be a natural number and $f\in L_\infty(\bA,\mathcal{A})$. Let
$$g(x,t):=\mathbb{E}_x((\Delta_t f)(x)|\mathcal{F}_{k-1})$$
then
$$\mathcal{K}_k(f)(x,y)=g(y,x-y).$$
\end{lemma}

The above lemma follows from the results in \cite{Sz1}.
Let us assume that $f=f'+g$ where $f'=\mathbb{E}(f|\mathcal{F}_k)$.
Then we have the following.

\begin{lemma} For almost every $t$ we have that
$$\mathbb{E}_x((\Delta_t f)(x)|\mathcal{F}_{k-1})=\mathbb{E}_x((\Delta_t f')(x)|\mathcal{F}_{k-1}).$$
\end{lemma}

\begin{proof} Let $W_k$ denote the orthogonal space of $L_2(\mathcal{F}_k)$. It is clear that $W_k$ is a shift invariant $L_\infty(\mathcal{F}_k)$ module. We have that $$(\Delta_t f)(x)=(\Delta_t f')(x)+(\Delta_t g)(x)+f'(x+t)\overline{g(x)}+g(x+t)\overline{f'(x)}$$
where the last two terms are in $W_k$. This implies that
\begin{equation}\label{speq1}
\mathbb{E}_x((\Delta_t f)(x)|\mathcal{F}_{k-1})=\mathbb{E}_x((\Delta_t f')(x)|\mathcal{F}_{k-1})+\mathbb{E}((\Delta_t g)(x)|\mathcal{F}_{k-1}).
\end{equation}
It remains to prove that the second term in (\ref{speq1}) is $0$ for almost every $t$.
In other words we have to prove that
$\|\Delta_t g\|_{U_k}=0$ hold almost always.
This follows from the next equation.
$$\int\|\Delta_t g\|^{2^k}_{U_k}~dt=\|g\|_{U_{k+1}}^{2^{k+1}}=0.$$
\end{proof}

An important corollary is the following.

\begin{corollary}\label{specor} For an arbitrary $f\in L_\infty(\bA,\mathcal{A})$ the equation $\mathcal{K}_k(f)=\mathcal{K}_k(\mathbb{E}(f|\mathcal{F}_k))$ holds.
\end{corollary}

Two variable measurable functions in general can be interpreted as integral kernel operators.
It will be fruitful to see $\mathcal{K}_k$ as an operator valued operator. It is obvious from the definition that $\mathcal{K}_k(f)$ is a self adjoint operator.
Due to the fact that $\mathcal{K}_k(f)$ is measurable in $\mathcal{F}_k\times\mathcal{F}_k$, for every $f$ there is a separable sub $\sigma$-algebra $\mathcal{B}\subset\mathcal{F}_k$ such that $\mathcal{K}_k(f)$ is measurable in $\mathcal{B}\times\mathcal{B}$. This means that we can apply the theory of integral kernel operators on standard Lebesgue spaces.
We say that $\mathcal{K}_k(f)$ is {\bf simple} if all its non zero eigenvalues have multiplicity one.
For the sake of simplicity we introduce the convention that elements in $L_2(\bA)$ can be looked at ``column vectors''. With this convention, if $f,g\in L_2(\bA)$ then $fg^*$ denotes a two variable function defined by $fg^*(x,y)=f(x)\overline{g(x)}$.

The next crucial theorem shows that the $k$-th order Fourier decomposition of $f$ is ``basically'' the same as the spectral decomposition of $\mathcal{K}_k(f)$

\begin{theorem}[Spectral form of the fundamental theorem]\label{spec} The terms $\{f_i\}_{i=1}^\infty$ in the $k$-th order Fourier decomposition of $f$ satisfy
$$\mathcal{K}_k(f)=\sum_{i=1}^\infty f_if_i^*$$
converging in $L_2$.
\end{theorem}

\medskip

\begin{corollary}\label{specor2}The terms $\{f_i\}_{i=1}^\infty$ in the $k$-th order Fourier decomposition of $f$ are all eigenvectors of the operator $\mathcal{K}_k(f)$.
The eigenvalue corresponding to $f_i$ is $\|f_i\|_2^2$.
Furthermore the system $\{f_i\}_{i=1}^\infty$ $L_2$-spans the image space of $\mathcal{K}_k(f)$.
In particular if $\mathcal{K}_k(f)$ is simple then its eigenvectors are exactly the scalar multiples of the functions $f_i$.
\end{corollary}

\begin{proof} Thanks to corollary \ref{specor} we can assume that $f\in\mathcal{F}_k$ and $f=\sum_i f_i$ is the $k$-th order Fourier decomposition.
First of all notice that the two variable functions $f_if_j^*$ are pairwise orthogonal in $L_2(\bA\times\bA)$.
It is also clear that $\|f_if_j^*\|_2=\|f_i\|_2\|f_j\|_2$.
Now we have that
$$\sum_{i,j}\|f_if_j^*\|_2^2=\bigl(\sum_i\|f_i\|_2^2\bigr)^2=\|f\|_2^4<\infty$$
which implies (together with orthogonality) that $\sum_{i,j}f_if^*_j$ converges in $L_2$ to $ff^*$.
Let us write $ff^*=q+h$ where $q=\sum_i f_if_i^*$ and $h=\sum_{i,j,i\neq j}f_if_j^*$.
The function $q$ is an element of the space $D_k$ since the functions $f_i$ satisfy $\Delta_t f_i\in\mathcal{F}_{k-1}$ for every $t$.
We claim that $h$ is orthogonal to $D_k$. It is enough to prove that for every pair $i\neq j$ of natural numbers $f_if_j^*$ is orthogonal to $D_k$.
Let $g$ be an arbitrary function in $D_k$.
Then
$$(f_if_j^*,g)=\int_{x,y}f_i(x)\overline{f_j(y)}\overline{g(x,y)}=\int_{x,t}f_i(x)\overline{f_j(x+t)}\overline{g(x,x+t)} .$$
Since $f_i(x)\overline{f_j(x+t)}$ is orthogonal to $L_2(\mathcal{F}_{k-1})$ for every fixed $t$ (see lemma \ref{ort2}) the proof is complete.
\end{proof}

\medskip

The natural question arises how to identify the $k$-th order decomposition of $f$ if $\mathcal{K}_k(f)$ is not simple. The following observation helps.

\begin{lemma}\label{findimpart} Assume that $F=\{f_1,f_2,\dots\}$ is the set of $k$-th degree Fourier components of a function $f$. Let $\lambda\neq 0$ be an eigenvalue of $\mathcal{K}_k(f)$, $V_\lambda$ be the corresponding eigen-subspace and let $F_\lambda=\{h|h\in F,\|h\|_2^2=\lambda\}$. Then $F_\lambda$ is an orthogonal basis of $V_\lambda$. Furthermore if $v=\sum_{h\in F_\lambda}\alpha_hh\in V_\lambda$
and $w=z+\sum_{h\in F_\lambda}\beta_hh$ where $(z,V_\lambda)=0$ then  $$\mathcal{K}_k(v)w=\sum_{h\in F_\lambda}|\alpha_h|^2\lambda\beta_h h.$$
In particular the (non zero) eigenvalues of $\mathcal{K}_k(v)$ are $\{|\alpha_h|^2\lambda~|h\in F_\lambda\}$ and the set $A_\lambda=\{\mathcal{K}_k(v)|v\in V_\lambda\}$ is a commutative algebra of dimension $d=\dim(V_\lambda)$ over $\mathbb{R}$.
\end{lemma}

\begin{proof} This lemma follows from theorem \ref{spec} by using the fact that the $k$-th order components of $v$ are exactly $\{\alpha_h h|h\in F_\lambda\}$ and so $\mathcal{K}_k(v)=\sum_{h\in F_\lambda}|\alpha_h|^2hh^*$.
\end{proof}

\begin{corollary} If $v\in V_\lambda$ is such that the eigenvalues of $\mathcal{K}_k(v)$ are all distinct then the one dimensional eigen-subspaces are generated by the elements from $F_\lambda$.
\end{corollary}

Now assume that we are given the subspace $V_\lambda$ and we want to determine the set $F_\lambda$.
First of all note that if $h'$ is a scalar multiple of an element $h\in F_\lambda$ then $h=h'(h',f)/\|h'\|_2^2$.
The previous corollary shows that, to find the elements of $F_\lambda$ up to scalar multiples, it is enough to find a vector $v\in V_\lambda$ with distinct eigenvalues. We also know by lemma \ref{findimpart} that almost all vectors in $V_\lambda$ have this property.
It will be useful to give a ``finite'' algorithm to find such a vector $v$.
The next lemma says that there is finite set of vectors that can be used for this purpose.

\begin{lemma}\label{sepset} For every $n$ there a finite subset $T_n$ of the $n$-dimensional euclidean space $E_n$ over $\mathbb{C}$ such that for every orthonormal basis $b_1,b_2,\dots,b_n$ of $E_n$ there is a vector $v\in T_n$ with $||(v,b_i)|-|(v,b_j)||>1$ whenever $i\neq j$.
\end{lemma}

\begin{proof} If we don't require the finiteness of $T_n$ then the compact set $$T'_n=\{v|\|v\|^2_2\leq 4n^3\}$$ is clearly good since $v=\sum_{i=1}^n 2ib_i$ is a good choice. The compactness of $T'_n$ implies that for every $\epsilon>0$ there is a finite subset $S_\epsilon$ in $T'_n$ such that every vector in $T'_n$ is of distance at most $\epsilon$ form this set. It is clear that if $\epsilon$ is small enough then $S_\epsilon$ is a good choice for $T_n$.
\end{proof}

\medskip

Now we summarize our knowledge:

\medskip

\noindent{\bf An algorithm to compute the $k$-th order Fourier decomposition:}

\medskip

Let $\lambda_1>\lambda_2>\lambda_3\dots$ be the list of distinct eigenvalues of $\mathcal{K}_k(f)$. Let $V_1,V_2,V_3,\dots$ be the list of corresponding eigenspaces with $\dim(V_i)=d_i$. For each space $V_i$ we choose a separating set $T_{d_i}$ guaranteed by lemma \ref{sepset}. Then we pick an element $v_i\in T_{d_i}$ such that $\mathcal{K}_k(v_i)$ has $d_i$ distinct eigenvalues. We denote the set of eigenspaces of $\mathcal{K}_k(v_i)$ by $Q_i$.
We have that $Q=\cup_i Q_i$ is the collection of one dimensional subspaces generated by the $k$-th order Fourier components of $f$. Let $\{w_q\}_{q\in Q}$ be a system of functions with $w_q\in q~,~\|w_q\|_2=1$. Then the set $\{f_q=w_q(w_q,f)|q\in Q\}$ is the set of $k$-th order Fourier components of $f$.

\subsection{Limits of self adjoint operators}\label{limop}

Let $M_i:X_i\times X_i\rightarrow \mathbb{C}$ be a sequence of self adjoint matrices with $\|M_i\|_\infty\leq 1$ and $|X_i|=m_i$. As always in this paper, we think of $X_i$ as a finite probability space with the uniform distribution and $L_2(X_i)$ as $L_2$ of this probability space. This defines the proper normalization of scalar products and action on $L_2(X_i)$. In particular we think of $M_i$ as an integral kernel operator on $L_2(X_i)$. This means that if $f:X_i\rightarrow\mathbb{C}$ is a function then then the product $M_if$ is defined by.
$$M_if(x)=m_i^{-1}\sum_{y\in X_i}M_i(x,y)f(y).$$ Traces and products of matrices are normalized accordingly.

Assume that the eigenvalues (with multiplicities) of $M_i$ are $\lambda_{i,1},\lambda_{i,2},\dots$ ordered in a way that $|\lambda_{i,1}|\geq|\lambda_{i,2}|\geq\dots$. Let $f_{i,j}$ be an eigenvector corresponding to $\lambda_{i,j}$ such that for every fixed $i$ the system $\{f_{i,j}\}_{j=1}^{m_i}$ is an orthonormal basis. We have that
$$M_i=\sum_{i=1}^{m_i}f_{i,j}f^*_{i,j}\lambda_{i,j}.$$

\begin{theorem}\label{limop} Let $M$ denote the ultra limit of $\{M_i\}_{i=1}^\infty$, $f_j$ denote the ultra limit of $\{f_{i,j}\}_{i=1}^\infty$ and $\lambda_j$ denote the ultra limit of $\{\lambda_{i,j}\}_{j=1}^\infty$. Then
\begin{equation}\label{oplimeq}
\mathbb{E}(M|\sigma_{1}\times\sigma_{2})=\sum_{\{j|\lambda_j>0\}}^\infty f_jf_j^*\lambda_j
\end{equation}
converging in $L_2$, where $\sigma_{1}$ and $\sigma_{2}$ are the two cylindric $\sigma$-algebras on $\bX\times\bX$ depending on the first and second coordinates.
\end{theorem}

\begin{proof} We start with a few observations. First of all note that $\|f_{i,j}\|_\infty\leq 1/\lambda_{i,j}$. This implies that if $\lambda_j>0$ then the ultra limit of $f_j$ exists and it is an $L_\infty$ measurable function with $\|f_j\|_\infty\leq 1/\lambda_j$.

Second observation is that $\sum_{j=1}^\infty |\lambda_j|^2\leq 1$. This follows from the fact that for every $i$ we have $\sum_{j=1}^{m_i}|\lambda_j|^2=\tr(M_iM_i^*)\leq 1$.
In particular we get the $L_2$ convergence of the sum in the theorem.

Third observation is that for every $0-1$ valued function $h$ we have that $$|hM_ih^*-\sum_{j=1}^t h^*f_{i,j}f^*_{i,j}h\lambda_{i,j}|\leq \sqrt{1/t}.$$ To see this we rewrite the left hand side as
$$|\sum_{j=t+1}^{m_i} h^*f_{i,j}f^*_{i,j}h\lambda_{i,j}|.$$ The term $h^*f_{i,j}f^*_{i,j}h$ is equal to $|(h,f_{i,j})|^2$. Since $\{f_{i,j}|t<j\leq m_i\}$ is an orthonormal system we have
$$\sum_{j=t+1}^{m_i} |(h,f_{i,j})|^2\leq \|h\|_2^2\leq 1.$$  The second observation implies that $|\lambda_{i,j}|\leq\sqrt{1/t}$ for all $j>t$. This completes the proof of the third observation.

Now we are ready to finish the proof. Let $M'$ be the left hand side of (\ref{oplimeq}). The third observation implies that if $h$ is an ultra limit of $0-1$ valued functions than $h^*Mh=h^*M'h$. It is easy to see that this implies that for two different measurable $0-1$ valued functions $h_1$ and $h_2$ we have $h_1^*Mh_2=h_1^*M'h_2$. In other words the integral of $M$ on product sets is the same as the integral of $M'$. Since every measurable set in $\sigma_1\times\sigma_2$ can be approximated by disjoint unions of product sets and $M'$ is obviously measurable in $\sigma_1\times\sigma_2$ the proof is complete.
\end{proof}

The above theorem has interesting consequences.

\begin{theorem}[Spectral form of the regularity lemma] For every function $F:\mathbb{N}\rightarrow\mathbb{R}^+$ and every $\epsilon>0$ there is a constant $n$ such that for every self adjoint matrix $M:V\times V\rightarrow\mathbb{C}$ with $\|M\|_\infty\leq 1$ and $|V|>n$ there is a number $n_0<n$ such that the spectral approximation
$\sum_{i=1}^{n_0}f_if_i^*\lambda_i$ of $M$ is at most $\epsilon$ close in $L_2$ to a matrix $M'$ which is $F(n_0)$ close in the normalized cut norm to $M$.
\end{theorem}

Another consequence (roughly speaking) says that if $G_i$ is a convergent graph sequence with limit $W$ then for every fixed $k$ the joint distribution of the values in the first $k$ eigenvectors (corresponding to the $k$ biggest eigenvalues) converges in distribution to the joint distribution of the first $k$ eigenvectors of $W$ (as an integral kernel operator).
Note that there is a small problem with multiple eigenvalues which can be handled by careful formulation.
These results and many others are proved in the paper \cite{Sz2}.

\subsection{Measurable Homomorphisms}

Let $h_i:A_i\rightarrow B_i$ be a sequence of surjective homomorphisms between the finite Abelian groups $A_i$ and $B_i$. Let $\bB$ denote the ultra product of the sequence $\{B_i\}_{i=1}^\infty$ and $\mathcal{B}$ the ultra product $\sigma$-algebra on it. We can take the ultra product $\bh$ of the functions $h_i$ which creates a surjective measure preserving homomorphism from $\bA$ to $\bB$. We will need a few lemmas. The next one is obvious.

\begin{lemma}\label{normhom} If $f$ is a function in $L_\infty(\mathcal{B})$ and $k$ is a natural number then
$$U_k(f)=U_k(\bh\circ f)$$
where the first Gowers norm is computed on $\bB$ and the second is computed on $\bA$.
\end{lemma}

\begin{lemma}\label{projnorm} If $f\in L_\infty(\mathcal{A})$ satisfies $U_k(f)=0$ then $U_k(\mathbb{E}(f|\bh^{-1}(\mathcal{B})))=0$.
\end{lemma}

\begin{proof} First of all note that $\bh^{-1}(\mathcal{B})$ is a coset $\sigma$-algebra on $\bA$ so it is perpendicular to every shift invariant $\sigma$-algebra. Now $U_k(f)=0$ implies that $\mathbb{E}(f|\mathcal{F}_k)$ is the $0$ function. This implies that $$\mathbb{E}(\mathbb{E}(f|\mathcal{F}_k)|\bh^{-1}(\mathcal{B}))=0.$$
Using perpendicularity we can interchange the projections and obtain
$$\mathbb{E}(\mathbb{E}(f|\bh^{-1}(\mathcal{B}))|\mathcal{F}_k)=0.$$
This completes the proof.
\end{proof}

As a consequence we get the following.

\begin{lemma} $\bh^{-1}(\mathcal{F}_k(\bB))=\mathcal{F}_k(\bA)\cap\bh^{-1}(\mathcal{B})$.
\end{lemma}

\begin{proof} Let $f\in L_\infty(\bh^{-1}(\mathcal{B}))$. We know that $f\in\mathcal{F}_k$ if and only if $f$ is orthogonal to any function $g$ with $U_k(g)=0$. However the scalar product $(f,g)$ is equal to $(f,\mathbb{E}(g|\bh^{-1}(\mathcal{B})))$. By Lemma \ref{projnorm} this implies that $f\in\mathcal{F}_k$ if and only if $f$ is orthogonal to every $g$ with $U_k(g)=0$ such that $g$ is measurable in $\bh^{-1}(\mathcal{B})$. This is equivalent with saying that $f$ has the form $\bh\circ f'$ where $f'$ is measurable in $\mathcal{F}_k(\bB)$.
\end{proof}

\subsection{Lemmas on $\sigma$-algebars}

In this section $(X,\mathcal{A},\mu)$ will be a probability space with $\sigma$-algebra $\mathcal{A}$ and measure $\mu$.
All the $\sigma$-algebras occurring in this chapter will be assumed to be subalgebras of $\mathcal{A}$.

For two $\sigma$-algebras $\mathcal{B}$ and $\mathcal{C}$ in $\mathcal{A}$ we denote by $\mathcal{B}\vee\mathcal{C}$ the smallest $\sigma$-algebra containing them and by $\mathcal{B}\wedge\mathcal{C}$ the intersection of them.

\begin{definition} We say that $\mathcal{B}$ and $\mathcal{C}$ are {\bf perpendicular} if for any $\mathcal{B}$ measurable function $f$ the function $\mathbb{E}(f|\mathcal{C})$ is again $\mathcal{B}$ measurable.
\end{definition}

We will use the following lemma.

\begin{lemma}\label{modul} Let $\mathcal{B}$ and $\mathcal{C}$ be two perpendicular $\sigma$-algebras and let $\mathcal{B}_1$ be a sub $\sigma$-algebra of $\mathcal{B}$.
Then $(\mathcal{C}\vee\mathcal{B}_1)\wedge\mathcal{B}=(\mathcal{C}\wedge\mathcal{B})\vee\mathcal{B}_1$.
\end{lemma}

\begin{proof}
Both terms of the right hand side are contained in both terms of the left hand side so we have that $(\mathcal{C}\vee\mathcal{B}_1)\wedge\mathcal{B}\supseteq(\mathcal{C}\wedge\mathcal{B})\vee\mathcal{B}_1$.
To see the other containment let $H$ be a set in $(\mathcal{C}\vee\mathcal{B}_1)\wedge\mathcal{B}$.
Using that $H\in\mathcal{C}\vee\mathcal{B}_1$ we have that for an arbitrary small $\epsilon>0$ there is an approximation of the characteristic function $1_H$ in $L_2$ of the form $$f=\sum_{i=1}^n 1_{C_i}1_{B_i}$$
where $C_i\in\mathcal{C}$ and $B_i\in\mathcal{B}_1$. Since $\|1_H-f\|_2\leq\epsilon$ and $H\in\mathcal{B}$ we have that $$\epsilon\geq\|\mathbb{E}(1_H|\mathcal{B})-\mathbb{E}(f|\mathcal{B})\|_2=\|1_H-\sum_{i=1}^n \mathbb{E}(1_{C_i}|\mathcal{B})1_{B_i}\|_2.$$
Using perpendicularity we have that $E(1_{C_i}|\mathcal{B})$ is measurable in $\mathcal{B}\wedge\mathcal{C}$ for every $i$ and the whole sum is measurable in $(\mathcal{C}\wedge\mathcal{B})\vee\mathcal{B}_1$.
Using it for every $\epsilon$ the proof is complete.
\end{proof}

\subsection{Cubic structure}

Let $d$ be a fixed natural number. For every subset $S\subseteq\{1,2,\dots,d\}$ we introduce a homomorphism $\psi_S:\bA^{d+1}\rightarrow\bA$ defined by
$$\psi_S(x,t_1,t_2,\dots,t_d)=x+\sum_{i\in S}t_i.$$
These functions are measurable in the ultra product measure.
Let us denote by $\mathcal{B}_S$ the $\sigma$-algebra $\psi_S^{-1}(\mathcal{A})$.
The next crucial fact was proved in \cite{Sz1}.

\begin{lemma}\label{cubint} $$\mathcal{B}_H\bigwedge\Bigl(\bigvee_{S\neq H\subseteq [d]}\mathcal{B}_S\Bigr)=\psi_H^{-1}(\mathcal{F}_{d-1})$$
when $H$ is any fixed subset of $[d]$.
\end{lemma}

The reader should also notice that the $\sigma$-algebras $\mathcal{B}_S$ behave in a very symmetric way. There is a group of measurable automorphisms that permute these $\sigma$-algebras transitively. In fact this group is the same as the automorphism group of the $d$-dimensional cube.
The details are given in \cite{Sz1}.
Note that the $\sigma$-algebras $\mathcal{B}_S$ are so called coset $\sigma$-algebras (see \cite{Sz1}) and so they are perpendicular to any shift invariant $\sigma$-algebra.
An important corollary of Lemma \ref{cubint} is the following.

\begin{corollary}\label{zeroint} Let $\{f_S\}_{S\subseteq [d]}$ be a collection of $L_\infty(\bA,\mathcal{A})$ functions. Assume furthermore that there exists a subset $H\subseteq [d]$ such that $\mathbb{E}(f_H|\mathcal{F}_{d-1})$ is the zero function. Then

\begin{equation}\label{zeroineq}
\int_{x,t_1,t_2,\dots,t_d}~~\prod_{S\subseteq [d]}f_S\bigl(x+\sum_{i\in S}t_i\bigl)=\int_{x,t_1,t_2,\dots,t_d}~~\prod_{S\subseteq [d]}f_S\circ\psi_S=0.
\end{equation}

\end{corollary}

\begin{proof} Let $\mathcal{B}$ denote the shift invariant $\sigma$-algebra
$$\bigvee_{S\neq H\subseteq [d]}\mathcal{B}_S.$$
We have that the integral (\ref{zeroineq}) can be rewritten in the form

\begin{equation}\label{zeroineq2}
\int_{x,t_1,t_2,\dots,t_d}~~\mathbb{E}(f_H\circ\psi_H|\mathcal{B}\bigr)\prod_{H\neq S\subseteq [d]}f_S\circ\psi_S.
\end{equation}

because the product $\prod_{H\neq S\subseteq [d]}f_S\circ\psi_S$ is measurable in $\mathcal{B}$.

Using lemma \ref{cubint} and that $\psi_H^{-1}(\mathcal{A})$ is perpendicular to $\mathcal{B}$ we obtain that

$$\mathbb{E}(f_H\circ\psi_H|\mathcal{B})=\mathbb{E}(f_H\circ\psi_H|\mathcal{B}\cap\mathcal{B}_H)=$$
$$=\mathbb{E}(f_H\circ\psi_H|\psi_H^{-1}(\mathcal{F}_{d-1}))=
\mathbb{E}(f_H|\mathcal{F}_{d-1})\circ\psi_H=0.$$

This implies that (\ref{zeroineq2}) is zero.

\end{proof}

We now prove a useful lemma.

\begin{lemma}[Mixed term lemma]\label{mixedterm} Let $\{f_S\}_{S\subseteq [d+1]}$ be a system of $L_\infty$ functions on $\bA$. Suppose that there is a collection of shift invariant subsets $\mathcal{H}_1,\mathcal{H}_2,\dots,\mathcal{H}_n$ in $L_\infty$ of $\bA$ such that for every pair $1\leq i<j\leq n$ and functions $g_1\in\mathcal{H}_i,~g_2\in\mathcal{H}_j$ we have $\mathbb{E}(g_1\overline{g_2}|\mathcal{F}_{d-1})=0$. Assume furthermore that $f_S$ is contained in one of these sets for every $S$. Then
$$\mathbb{E}_{x,t_1,t_2,\dots,t_{d+1}}\prod_{S\subseteq [d+1]}(f_S\circ\psi_S)^{c(|S|)}=0$$ whenever the functions $f_S$ are not all contained in the same set. (The operation $c(|S|)$ is the complex conjugation applied $|S|$-times.)
\end{lemma}

\begin{proof} If the functions $f_S$ are not all contained in the same set then there is an element $r\in [d+1]$ and $H\subseteq [d+1]$ with $r\notin H$ such that $f_H$ and $f_{H\cup\{r\}}$ are not in the same set. Without loss of generality we assume that $r=d+1$.
Now for $S\subseteq [d]$ and fixed value of $t_{d+1}$ let $x\rightarrow f'_{S,t_{d+1}}(x)$ denote the function $f_S(x)\overline{f_{S\cup\{d+1\}}(x+t_{d+1})}$. Using the shift invariance of the sets we have that $\mathbb{E}(f'_{H,t_{d+1}}|\mathcal{F}_{d-1})=0$. Using lemma \ref{zeroint} for this system of functions we obtain that
$$\mathbb{E}_{x,t_1,t_2,\dots,t_d}\prod_{S\subseteq [d]}(f'_{S,t_{d+1}}\circ\psi_S)^{c(|S|)}=0$$
for every fixed $t_{d+1}$. This completes the proof.
\end{proof}

\begin{lemma}\label{twofunct} Let $f,g$ be $L_\infty(\mathcal{A})$ functions such that $f=f_1+f_2$ where $f_1=\mathbb{E}(f|\mathcal{F}_{d-1})$ and $g=g_1+g_2$ where $g_1=\mathbb{E}(g|\mathcal{F}_{d-1})$. Then
\begin{equation}\label{twofun} \int_t\|f(x)\overline{g(x+t)}\|^{2^d}_{U_d}=\int_t\|f_1(x)\overline{g_1(x+t)}\|^{2^d}_{U_d}+\|f_2(x)\overline{g_2(x+t)}\|^{2^d}_{U_d}.
\end{equation}
\end{lemma}

\begin{proof}
\begin{equation}\label{twofunct1}
\int_t\|f(x)\overline{g(x+t)}\|^{2^d}_{U_d}=\mathbb{E}_{x,t_1,t_2,\dots,t_{d+1}}\prod_{S\in [d+1]}(f_S\circ\psi_S)^{c(|S|)}
\end{equation}
Where $f_S=f$ whenever $d+1\notin S$, $f_S=g$ whenever $d+1\in S$ and $c(|S|)$ is the complex conjugation applied $|S|$-times.
By substituting $f=f_1+f_2$ and $g=g_1+g_2$ into (\ref{twofunct1}) we get $2^{2^{d+1}}$ terms of the form
\begin{equation}\label{twofunct2}
\mathbb{E}_{x,t_1,t_2,\dots,t_{d+1}}\prod_{S\in [d+1]}(f'_S\circ\psi_S)^{c(|S|)}
\end{equation}
where $f'_S$ is either $f_1$ or $f_2$ if $d+1\notin S$ and is ether $g_1$ or $g_2$ if $d+1\in S$.
We call such a term ``mixed'' if both numbers $1$ and $2$ occur in the indices.

Obviously it is enough to show that mixed terms are $0$.
Using the two shift invariant Hilbert spaces $L_2(\mathbb{F}_{d-1})$ and its orthogonal space the previous ``mixed term lemma'' finishes the proof.
\end{proof}

\begin{corollary}\label{nozer} Let $f,g$ be $L_\infty(\mathcal{A})$ functions such that non of $\mathbb{E}(f|\mathcal{F}_{d-1})$ and $\mathbb{E}(g|\mathcal{F}_{d-1})$ is the $0$ function. Then the two variable function $h(x,t)=\mathbb{E}_x(f(x)\overline{g(x+t)}|\mathcal{F}_{d-1})$ is not the $0$ function.
\end{corollary}

\begin{proof} The support of both $f_1=\mathbb{E}(f|\mathcal{F}_{d-1})$ and $g_1=\mathbb{E}(g|\mathcal{F}_{d-1})$ has positive measure. This means that for a positive measure set of $t$'s the supports of $f_1$ and $g_1(x+t)$ intersect each other in a positive measure set. Since $U_d$ is a norm on $L_\infty(\mathcal{F}_{d-1})$ we get that in (\ref{twofun}) the right hand side is not $0$. By lemma \ref{twofunct} it means that for a positive measure set of $t$'s $\|f(x)\overline{g(x+t)}\|_{U_d}\neq 0$. This means that $h(x,t)$ is not the $0$ function.
\end{proof}

\subsection{Gowers's norms in terms of character decomposition}

The main result of this chapter is an additivity theorem for the $2^k$-th power of the $U_k$ norm.

\begin{theorem}[Additivity] Let $\phi$ be an $L_\infty$ function with a decomposition
$\phi=\phi_1+\phi_2+\phi_3+\dots$ converging in $L_2$. Assume that $\mathbb{E}(\phi_i(x+t)\overline{\phi_j(x)}|\mathcal{F}_{k-2})=0$ for every pair $i\neq j$ and $t\in\bA$. Then
$$\|\phi\|_{U_k}^{2^k}=\sum_{i=1}^\infty \|\phi_i\|_{U_k}^{2^k}$$
or equivalently
$$\mathbb{E}(\Delta_{t_1,t_2,\dots,t_k}\phi(x))=\sum_{i=1}^\infty\mathbb{E}(\Delta_{t_1,t_2,\dots,t_k}\phi_i(x))$$
where $t_1,t_2,\dots,t_k$ and $x$ are chosen uniformly at random.
\end{theorem}

\begin{proof} First of all note that the series $\phi_1+\phi_2+\dots$ converges in $L_2$. Since the $L_2$ norm is an upper bound on the Gowers norm it is enough to verify the statement for the case when this sum is finite. Let us assume that $\phi=\phi_1+\phi_2+\dots+\phi_n$.
We have that
$$\Delta_{t_1,t_2,\dots,t_k}\phi(x)=\prod_{S\subseteq [k]}\phi\bigl(x+\sum_{i\in S}t_i\bigr)^{c(|S|)}=\sum_{f:2^{[k]}\rightarrow [n]}Q_f(x,t_1,t_2,\dots,t_k)$$
where
$$
Q_f(x,t_1,t_2,\dots,t_k)=\prod_{S\subseteq [k]}\phi_{f(S)}\bigl(x+\sum_{i\in S}t_i\bigr)^{c(|S|)}.
$$
where $c(r)$ is $r$-th power of the conjugation operation on $\mathbb{C}$.
If $f$ is not a constant function then we say that $Q_f$ is a mixed term. We denote by $\mathcal{H}_i$ the set of functions $\{x\rightarrow \phi_i(x+t)|t\in\bA\}$. The mixed term lemma \ref{mixedterm} applied to these sets shows that the integrals of mixed terms are all
zero. This finishes the proof.
\end{proof}

\begin{corollary} Let $f$ be an $L_\infty$ function with $k-1$-order Fourier decomposition
$f=g+f_1+f_2+f_3+\dots$ where $f'=f_1+f_2+\dots=\mathbb{E}(f|\mathcal{F}_{k-1})$. Then
$$\|f\|_{U_k}^{2^k}=\|f'\|_{U_k}^{2^k}=\sum_{i=1}^\infty \|f_i\|_{U_k}^{2^k}.$$
\end{corollary}

\section{Higher order dual groups}

Let $\hat{\bA}_k$ denote the set of rank $1$ modules over $L_\infty(\mathcal{F}_{k-1},\bm)$.
The next lemma says that rank one modules are forming an abelian group. This fact is crucial
for higher order Fourier analysis.

\begin{lemma} The set $\hat{\bA}_k$ is an Abelian group with respect to point wise multiplication on $\bA$.
\end{lemma}

\begin{proof} Let $M_1$ and $M_2$ be two rank one modules over $L_\infty(\mathcal{F}_{k-1})$ and let $\phi_1\in M_1, \phi_2\in M_2$ be two $k$-th order characters. The product $\phi=\phi_1\phi_2$ satisfies $\Delta_t\phi=\Delta_t\phi_1\Delta_t\phi_2\in L_\infty(\mathcal{F}_{k-1})$ for every $t\in\bA$ and so it is contained in a rank one module $M_3$.
Any other functions $f_1\in M_1$ and $f_2\in M_2$ are of the form $f_1=\phi_1\lambda_1,~f_2=\phi_2\lambda_2$ for some functions $\lambda_1,\lambda_2\in L_\infty(\mathcal{F}_{k-1})$. This implies that $f_1f_2$ is in $M_3$.
Associativity and commutativity is clear. The inverse of $M_1$ is generated by $\overline{\phi_1}$
\end{proof}

\begin{definition} The abelian group structure on $\hat{\bA}_k$ is called the {\bf $k$-th order dual group} of $\bA$.
\end{definition}

The next theorem is the analogy of the fact from harmonic analysis that point wise multiplication of functions correspond to convolution in the dual language.

\begin{theorem}\label{highconv} Let $f,g$ be two functions in $L_\infty(\mathcal{F}_k)$. Assume for an arbitrary $a\in\hat{\bA}_k$ we denote the component of $f$ and $g$ in $a$ by $f_a$ and $g_a$. Then the component of $fg$ in $c\in\hat{\bA}_k$ is equal to
$$\sum_{a+b=c}f_ag_b$$
where the above sum has only countable many non zero term and the sum is convergent in $L_2$.
\end{theorem}

\begin{proof} First of all we observe that if $g$ is contained in one single rank one module then
the $k$-th order decomposition of $fg$ is $\sum_a f_ag$ since it converges in $L_2$ and the terms $f_ag$ are from distinct rank one modules.
From this observation we also get the statement if $g$ has finitely many non zero components.

If $g$ has infinitely many components then for an arbitrary $\epsilon$ we can approximate $g$  $\epsilon$-close in $L_2$ by a sub sum of its components $g_\epsilon$. Then $fg=fg_\epsilon+f(g-g_\epsilon)$. Here the $\|f(g-g_\epsilon)\|_2\leq \|f\|_\infty\epsilon$. So as $\epsilon$ goes to $0$ the $L_2$ error we make also goes to $0$.
\end{proof}

\subsection{Pure and locally pure characters}

The simplest example for a $k$-th order character is a function $\phi:\bA\rightarrow\mathcal{C}$ such that
$$\Delta_{t_1,t_2,\dots,t_{k+1}}\phi(x)=1$$
for every $t_1,t_2,\dots,t_{k+1},x$ in $\bA$.
Such functions will be called {\bf pure} characters.
Unfortunately for $k>1$ there are groups $\bA$ on which not every rank one module over $L_\infty(\mathcal{F}_{k-1})$ can be represented by a pure character.
It will turn out however that every $k$-th order character is (approximately) pure on ``neighborhoods'' measurable in $\mathcal{F}_{k-1}$. In other words $k$-th order characters are (approximately) ``locally pure'' modulo some $k-1$ degree decomposition of $\bA$.
To make these statements precise we need new notation.

\begin{definition} Let $\mathcal{B}\subseteq\mathcal{A}(\bA)$ be any $\sigma$ algebra.
We say that a function $\phi:\bA\rightarrow\mathcal{C}$ is a $\mathcal{B}$-locally pure character of degree $k$ if the function
$\Delta_{t_1,t_2,\dots,t_{k+1}}\phi(x)$ on $\bA^{k+2}$ is measurable in the $\sigma$ algebra generated by the $\sigma$-algebras $\psi_S^{-1}(\mathcal{B})$ where $S$ runs trough the subsets of $\{1,2,\dots,k+1\}$. We denote the set of $\mathcal{B}$ locally pure characters of degree $k$ by $[\mathcal{B},k]^*$.
\end{definition}

In case $\mathcal{B}$ is the trivial $\sigma$-algebra then $[\mathcal{B},k]^*$ is just the set $k$-th order pure characters.

\begin{lemma}\label{pureprop} Let $\mathcal{B}\subseteq\mathcal{A}(\bA)$ be a $\sigma$-algebra. Then: \begin{enumerate}
\item $[\mathcal{B},k]^*$ is an Abelian group with respect to point wise multiplication.
\item $[\mathcal{B},0]^*$ is the set of $\mathcal{B}$ measurable functions $f:\bA\rightarrow \mathcal{C}$
\item $[\mathcal{B},k]^*\subseteq [\mathcal{B},k+1]^*$
\item If $\mathcal{B}$ is invariant and $\phi\in[\mathcal{B},k]^*$ then $\Delta_t\phi\in[\mathcal{B},k-1]^*$ for every $t\in\bA$
\end{enumerate}
\end{lemma}

\begin{definition} We introduce the $\mathcal{B}$-local $k$-th order dual group $[\mathcal{B},k]^0$ as the factor group $$[\mathcal{B},k]^*/[\mathcal{B},k-1]^*.$$
\end{definition}

One motivation for this notion is that it will turn out that $\hat{\bA}_k$ is naturally isomorphic to $[\mathcal{F}_{k-1},k]^0$.

\begin{lemma}\label{charmes} Every element in $[\mathcal{B},k]^*$ is measurable in the $\sigma$-algebra  $\mathcal{F}_k\vee\mathcal{B}$.
\end{lemma}

\begin{proof} Assume that $f\in[\mathcal{B},k]^*$. Let us consider the probability space $\bA^{k+2}$. Let $\mathcal{C}$ be the $\sigma$-algebra generated by $\psi_S^{-1}(\mathcal{A})$ where $S$ runs through the non empty subsets of $\{1,2,\dots,k+1\}$.
Let $g_1=\Delta_{t_1,t_2,\dots,t_{k+1}}f(x)$ and $g_2=f(x)/g_1(x,t_1,t_2,\dots,t_{k+1})$.
We have that $g_1$ is measurable in $\mathcal{C}\vee\psi_\emptyset^{-1}(\mathcal{B})$ and $g_2$ is measurable in $\mathcal{C}$. This means that $f=g_1g_2$ is measurable in $\mathcal{C}\vee\psi_\emptyset^{-1}(\mathcal{B})$. On the other hand $f$ is measurable in $\psi_\emptyset^{-1}(\mathcal{A})$ and thus it is measurable in
$$(\mathcal{C}\vee\psi_\emptyset^{-1}(\mathcal{B}))\wedge\psi_\emptyset^{-1}(\mathcal{A}).$$
By definition $\psi_\emptyset^{-1}(\mathcal{A})$ is a coset $\sigma$-algebra and so it is perpendicular to any shift invariant $\sigma$-algebras. Now we can use lemma \ref{modul} and lemma \ref{cubint}
and obtain that $f$ is measurable in
$$(\mathcal{C}\wedge\psi_\emptyset^{-1}(\mathcal{A}))\vee\psi_\emptyset^{-1}(\mathcal{B})=\psi_\emptyset^{-1}(\mathcal{F}_k)\vee\psi_\emptyset^{-1}(\mathcal{B})=\psi_\emptyset^{-1}(\mathcal{F}_k\vee\mathcal{B}).$$

\end{proof}

\begin{lemma}\label{charsep} If $\phi$ is a $k$-th order character then $\phi\in[\mathcal{F}_{k-1},k]^*$.
\end{lemma}

\begin{proof} Let $f$ denote the function $\Delta_{t_1,t_2,\dots,t_{k+1}}\phi(x)$ on $\bA^{k+2}$. It is clear that for every $t\in\bA^{k+2}$ we have that $\Delta_t f$ is measurable in the $\sigma$-algebra $\mathcal{B}$ generated by $\psi_S^{-1}(\mathcal{F}_{k-1})$ where $S$ runs through the subsets of $\{1,2,\dots,k+1\}$. Using that $\mathcal{B}$ is shift invariant and lemma $\ref{ort1}$ we obtain that either $f$ is measurable in $\mathcal{B}$ or $E(f|\mathcal{B})=0$.
However the second possibility is impossible since the integral of $f$ is equal to the $k+1$-th Gowers norm of $\phi$ which is positive because $\phi\in\mathcal{F}_k$.
\end{proof}

\begin{corollary}\label{secder} If $\phi$ is a $k$-th order character then there is an invariant  $\sigma$-algebra $\mathcal{F}_{k-2}\subseteq\mathcal{B}\subset\mathcal{F}_{k-1}$ which is relative separable over $\mathcal{F}_{k-2}$ such that $\Delta_{t_1,t_2}\phi$ is measurable in $\mathcal{B}$ for every pair $t_1,t_2\in\bA$.
\end{corollary}

\begin{proof} Lemma \ref{charsep} implies that there is a separable $\sigma$-algebra $\mathcal{B}_1\subset\mathcal{F}_{k-1}$ such that $\phi\in [\mathcal{B}_1,k]^*$. Let $\mathcal{B}$ denote the smallest invariant $\sigma$ algebra containing $\mathcal{F}_{k-2}$ and $\mathcal{B}_1$.
Then $\phi\in[\mathcal{B},k]^*$. It is clear that $\mathcal{B}$ is relative separable over $\mathcal{F}_{k-2}$ since it is generated by the shifts of countable many relative separable elements over $\mathcal{F}_{k-2}$. Lemma \ref{pureprop} implies that $\Delta_{t_1,t_2}\phi\in[\mathcal{B},k-2]^*$ for every $t_1,t_2\in\bA$.
Using lemma \ref{charmes} we obtain that $\Delta_{t_1,t_2}\phi$ is measurable in $\mathcal{B}\vee\mathcal{F}_{k-2}=\mathcal{B}$.
\end{proof}

\subsection{$k$-types}

A {\bf $k$-type} $T$ is defined as a countable subgroup of $\hat{\bA}_k$.
It can be seen that $k$-types are in a one to one correspondence with the shift invariant $\sigma$-algebras $\mathcal{B}$ that can be generated by countable many sets and $\mathcal{F}_{k-1}$. Here $\mathcal{B}$ is the smallest $\sigma$-algebra in which all the elements from $T$ are measurable. The group $T$ is obtained from $\mathcal{B}$ as the rank $1$ decomposition of $L_2(\mathcal{B},\bm)$

\begin{definition} Let $f$ be a measurable function in $L_2(\mathcal{F}_k)$. We say that the $k$-th dual-support $S(f)\subseteq\hat{\bA}_k$ of $f$ is the set of rank one modules that are not orthogonal to $f$.
It is clear that $S(f)$ is a countable set. The {\bf $k$-type} $T(f)$ of $f$ is the group generated by the elements $\{g_1g_2^{-1}~|~g_1,g_2\in S(f)\}$. If the value of $k$ is not clear from the context then we will use $T_k(f)$ and $S_k(f)$ instead of $T(f)$ and $S(f)$.
\end{definition}

The next lemma follows from theorem \ref{highconv}

\begin{lemma}\label{prodsup} If $f,g\in L_\infty(\mathcal{F}_k)$ then $S(fg)\subseteq S(f)S(g)$.
\end{lemma}

\begin{lemma}\label{type1} Let $T$ be a fixed $k$-type and let $\mathcal{H}$ denote the set of functions in $L_\infty(\mathcal{F}_k,\bm)$ whose types are contained in $T$. Then the set $\mathcal{H}$ is closed under point-wise multiplication and under the shift operations. In particular $\Delta_t\mathcal{H}\subseteq \mathcal{H}$ for every $t\in\bA$.
\end{lemma}

\begin{proof} Closeness under shift operations is trivial. Lemma \ref{prodsup} implies the statement on the the point wise multiplication.
\end{proof}

\begin{proposition}\label{fixtype} Let $\phi$ be a $k$-th order character. Then there is a fixed $k-1$ type $T$ such that $T_{k-1}(\Delta_t\phi)\subseteq T$ for every fixed $t\in \bA$.
\end{proposition}

\begin{proof} Let $\mathcal{B}$ be the $\sigma$-algebra guaranteed by corollary \ref{secder} and let $T$ be the subgroup of $\hat{\bA}_k$ corresponding to $\mathcal{B}$. We have that the dual support of $\Delta_{t_1,t_2}\phi$ is contained in $T$ for every $t_1,t_2$. Let $t_1\in\bA$ be an arbitrary fixed element and let $\Delta_{t_1}\phi=f_1+f_2+\dots$ be the unique $k-1$-th order Fourier decomposition of $\Delta_{t_1}\phi$ into non zero functions. Assume that $\lambda_i\in\hat{\bA}_{k-1}$ is the module containing $f_i$ for every $i$. We have to show that $\lambda_i\lambda_j^{-1}\in T$ for every pair of indices $i,j$.
Let us choose a $k-1$-th order character $\phi_i$ from every module $\lambda_i$ and let $g_i$ denote $(\Delta_{t_1}\phi)\overline{\phi_i}$.
We have that $\mathbb{E}(g_i|\mathcal{F}_{k-2})$ is not $0$.
This means by lemma \ref{nozer} that for a positive measure $t_2$'s $\mathbb{E}(g_i(x)\overline{g_j(x+t_2)}|\mathcal{F}_{k-2})$ is not the $0$ function. On the other hand $g_i(x)\overline{g_j(x+t_2)}=(\Delta_{t_1,t_2}\phi(x))\overline{\phi_i(x)}\phi_j(x+t_2)$.
Here $\overline{\phi_i(x)}\phi_j(x+t_2)$ is an element from the module $\lambda_j\lambda_i^{-1}$.
If $\mathbb{E}(g_i(x)\overline{g_j(x+t_2)}|\mathcal{F}_{k-2})$ is not $0$ for some $t_2$ then the $\lambda_i\lambda_j^{-1}$ component of $\Delta_{t_1,t_2}\phi$ is not zero. It shows that $\lambda_i\lambda_j^{-1}\in T$.
\end{proof}

\begin{theorem}\label{charhom} For every $k$-th order character $\phi$ there is a countable subgroup $T\subset\hat{\bA}_{k-1}$ and a homomorphism $h:\bA\rightarrow\hat{\bA}_{k-1}/T$ such that the $k$-th dual support of $\Delta_t\phi$ is contained in the coset $h(t)$.
\end{theorem}

\begin{proof} Proposition \ref{fixtype} implies that there is a countable subgroup $T\subset\hat{\bA}_{k-1}$ such that the dual support of $\Delta_t\phi$ is contained in a coset of $T$ for every element $t\in\bA$. We denote this coset by $h(t)$. We have to show that $h$ is a homomorphism. This follows from the identity
$$\Delta_{t_1+t_2}\phi(x)=\Delta_{t_2}\phi(x+t_1)\Delta_{t_1}\phi(x)$$
using that the dual support of a function is shift invariant and theorem \ref{highconv} the proof is complete.
\end{proof}

\begin{lemma}\label{countriv} Let $k\geq 2$ and $\phi$ be a $k$-th order character such that there is a countable subgroup $T\subseteq \hat{\bA}_{k-1}$ with the property that the $k-1$-th dual support of $\Delta_t\phi$ is contained in $T$ for every $t\in\bA$. Then $\phi$ is measurable in $\mathcal{F}_{k-1}$ (or in other words $\phi$ represents the trivial module.)
\end{lemma}

\begin{proof} Let $\mathcal{B}$ be the $\sigma$-algebra generated by the modules in $T$. Since $T$ is countable we have that $\mathcal{B}$ is a separable extension of $\mathcal{F}_{k-2}$. Now Lemma \ref{sepex} implies that $\phi$ is contained in a shift invariant separable extension of $\mathcal{B}$. which is also a shift invariant separable extension of $\mathcal{F}_{k-1}$.
\end{proof}

\subsection{Various Hom-sets}

We will need some notation. For two Abelian groups $A_1$ and $A_2$ we denote by $\hom(A_1,A_2)$ the set of all homomorphism from $A_1$ to $A_2$. The set $\hom(A_1,A_2)$ is an Abelian group with respect to the point wise multiplication. Let $\aleph_0(A_2)$ denote the set of countable subgroups in $A_2$. The groups $\hom(A_1,A_2/T)$ where $T\in\aleph_0(A_2)$ are forming a direct system with the natural homomorphisms $\hom(A_1,A_2/T_1)\rightarrow\hom(A_1,A_2/T_2)$ defined when $T_1\subseteq T_2$.

\begin{definition} $\hom^*(A_1,A_2)$ is the direct limit of the direct system $$\{\hom(A_1,A_2/T)\}_{T\in\aleph_0(A_2)}$$
with the homomorphisms induced by embeddings on $\aleph_0(A_2)$.
\end{definition}

We describe the elements of $\hom^*(A_1,A_2)$.
Let $H(A_1,A_2)$ be the disjoint union of all the sets $\hom(A_1,A_2/T)$ where $T$ is some countable subgroup of $A_2$. If $h_1\in\hom(A_1,A_2/T_1)$ and $h_2\in\hom(A_1,A_2/T_2)$ are two elements in $H(A_1,A_2)$ then we say that $h_1$ and $h_2$ are equivalent if there is a countable subgroup $T_3$ of $A_2$ containing both $T_1$ and $T_2$ such that $h_1$ composed with $A_2/T_1\rightarrow A_2/T_3$ is the same as $h_2$ composed with $A_2/T_2\rightarrow A_2/T_3$.
The equivalence classes in $H(A_1,A_2)$ are forming an Abelian group that we denote by $\hom^*(A_1,A_2)$.

For two abelian groups let $\hom^c(A_1,A_2)$ denote the set of homomorphisms whose image is countable. We denote by $\hom^0(A_1,A_2)$ the factor $\hom(A_1,A_2)/\hom^c(A_1,A_2)$. If $T$ is a countable subgroup of $A_2$ then there is a natural embedding of $\hom^0(A_1,A_2/T)$ into $\hom^*(A_1,A_2)$ in the following way.
The set $\hom(A_1,A_2/T)$ is a subset of $H(A_1,A_2)$. It is easy to see that $\phi_1,\phi_2\in\hom(A_1,A_2/T)$ are equivalent if and only if they are contained in the same coset of $\hom^c(A_1,A_2)$. From the definitions it follows that

\begin{lemma}\label{homstarcup} $$\hom^*(A_1,A_2)=\bigcup_{T\in\aleph_0(A_2)} \hom^0(A_1,A_2/T).$$
\end{lemma}

\begin{corollary}\label{exp} If $A_1$ is of exponent $n$ then so is $\hom^*(A_1,A_2)$.
\end{corollary}

Note that an abelian group is said to be of exponent $n$ if the $n$-th power of every element is the identity.
There is a simplification of the situation when $A_2$ is of exponent $p$ for some prime number $p$.

\begin{lemma} If $p$ is a prime number and $A_2$ is of exponent $p$ then $\hom^*(A_1,A_2)=\hom^0(A_1,A_2)$.
\end{lemma}

\begin{definition} We say that an Abelian group is essentially torsion free if there are at most countably many finite order elements in it.
\end{definition}

\begin{lemma}\label{estors} If $A$ is essentially torsion free then $A/T$ is essentially torsion free whenever $T\in\aleph_0(A)$.
\end{lemma}

\begin{proof} Assume by contradiction that there are uncountably many finite order elements in $A/T$. Then there is a natural number $n$ and element $t\in T$ such that the set
$S=\{x~|~x\in A,~x^n=t\}$ is uncountable. Then for a fixed element $y\in S$ the set $Sy^{-1}$ is an uncountable set of finite order elements in $A$ which is a contradiction.
\end{proof}

\begin{lemma}\label{nullfree} If $A_2$ is essentially torsion free then $\hom^0(A_1,A_2)$ is torsion free.
\end{lemma}

\begin{proof} Assume by contradiction that there is an element $\tau\in\hom(A_1,A_2)$ and $n\in\mathbb{N}$ such that $\tau(A_1)$ is uncountable but $\tau^n(A_1)$ is countable.
Similarly to the proof of lemma \ref{estors} this means that there is a fixed element $t\in\tau^n(A_1)$ whose pre image under the map $x\rightarrow x^n$ is uncountable which is a contradiction.
\end{proof}

\begin{lemma}\label{estfree} If $A_2$ is essentially torsion free then $\hom^*(A_1,A_2)$ is torsion free.
\end{lemma}

\begin{proof} By lemma \ref{homstarcup} it is enough to prove that for every $T\in\aleph_0(A_2)$ the group $\hom^0(A_1,A_2/T)$ is torsion free. Lemma \ref{estors} implies that $A_2/T$ is essentially torsion free. Lemma \ref{nullfree} finishes the proof.
\end{proof}

\subsection{On the structure of the higher order dual groups}

In this section we study the structure of higher order dual groups of $\bA$.
Let us start with $\hat{\bA}_1$.
We know that $\hat{\bA}_1$ is the group of measurable homomorphisms $\bA\rightarrow\mathcal{C}$.

\begin{lemma}\label{firstdual} The group $\hat{\bA}_1$ is isomorphic to $\bA$.
\end{lemma}

\begin{proof}
We have that $\bA$ is the ultra limit of a sequence $\{A_i\}_{i=1}^\infty$ of finite abelian groups.
Let $H\subseteq\hat{\bA}_1$ denote the set of those characters that are ultra limit of characters on $\{A_i\}_{i=1}^\infty$.
Let $\lambda_i:A_i\rightarrow\mathcal{C}$ and $\mu_i:A_i\rightarrow\mathcal{C}$ be two sequences of linear characters. Let furthermore $\lambda$ be the ultra limit of $\{\lambda_i\}_{i=1}^\infty$ and $\mu$ be the ultra limit of $\{\mu_i\}_{i=1}^\infty$.
If $\lambda_i$ differs from $\mu_i$ on an index set which is in the ultra filter then they are orthogonal at this index set and so they are orthogonal in the limit. This implies that $\lambda=\mu$ if and only if the sequences $\{\lambda_i\}_{i=1}^\infty$ and $\{\mu_i\}_{i=1}^\infty$ agree on a set from the ultra filter. In other words $H$ is isomorphic to the ultra product of the dual groups of $\{A_i\}_{i=1}^\infty$. Since the dual group of a finite abelian group $A$ is isomorphic to $A$ we obtain that $H$ is isomorphic to $\bA$.

To complete the proof we need to see that $H=\hat{\bA}_1$. This follows from the well known fact that approximate characters on a finite Abelian group can be approximated with proper characters.
\end{proof}

\medskip

By theorem \ref{charhom} every $k$-th order character $\phi$ induces a homomorphism from $\bA$ to $\hat{\bA}_{k-1}/T$ for some countable subgroup. This homomorphism represents an element in $\hom^*(\bA,\hat{\bA}_{k-1})$. We denote this element by $q(\phi)$. If $q(\phi_1)=q(\phi_2)$ then Lemma \ref{countriv} shows that if $k\geq 2$ then $\phi_1$ and $\phi_2$ belong to the same rank one module.
This implies the following lemma.

\begin{lemma}\label{dualemb} Let $k\geq 2$ and $\phi_1$ and $\phi_2$ be two $k$-th order characters. Them $q(\phi_1)=q(\phi_2)$ if and only if $\phi_1$ and $\phi_2$ generate the same rank one module over $L_\infty(\mathcal{F}_{k-1})$. It follows that $q$ is an injective homomorphism of $\hat{\bA}_k$ into $\hom^*(\bA,\hat{\bA}_{k-1})$
\end{lemma}

Of course if $k=1$ then $q$ is not injective however the situation is even better. In this case $\hat{\bA}_1$ is embedded into $\hom(\bA,\mathcal{C})=\hom(\bA,\hat{\bA}_0)$.
Lemma \ref{dualemb} has a few interesting consequences.

\begin{lemma}\label{duexp} Let $p$ be a prime number and assume that $\bA$ has exponent $p$. Then for every $k\geq 1$ the group $\hat{\bA}_k$ has exponent $p$.
\end{lemma}

\begin{proof} We go by induction. We know that $\hat{\bA}_1$ is isomorphic with $\bA$ so it has exponent $p$. Assume that it is true for $k$ then Corollary \ref{exp} and Lemma \ref{dualemb} ensure that it remains true for $k+1$.
\end{proof}

The next lemma follows immediately from Lemma \ref{dualemb}

\begin{lemma} $\hat{\bA}_k$ is embedded into both
$$\hom^*(\bA,\hom^*(\bA,\dots,\hom^*(\bA,\hat{\bA}_1))\dots)$$
and
$$\hom^*(\bA,\hom^*(\bA,\dots,\hom^*(\bA,\hom(\bA,\mathcal{C})))\dots )$$
where the number of $\hom^*$-s is $k-1$,
\end{lemma}

\begin{proof} The proof follows directly from Lemma \ref{dualemb} and the fact that $\hom^*(A_1,A_2)\subseteq \hom^*(A_1,A_3)$ whenever $A_2\subseteq A_3$.
\end{proof}

It interesting to note that
$$\hom(\bA,\hom(\bA,\dots,\hom(\bA,\mathcal{C}))\dots)\simeq \hom\Bigl(\bigotimes_{i=1}^k\bA,\mathcal{C}\Bigr)$$
which shows some connection of higher order Fourier analysis with higher order polynomials.
However the situation is more complex due to the presence of the $*$'s at the $\hom$ functions.

\subsection{Maps induced by measurable homomorphisms of $\bA$}

Let $\sigma_i:A_i\rightarrow A_i$ be a sequence of automorphisms. We denote by $\sigma$ the ultra product of $\{\sigma_i\}_{i=1}^\infty$ which is an automorphism of $\bA$. Automorphisms that arise this way will be called {\bf measurable automorphisms}. Similarly we can define measurable endomorphisms.

In the following lemmas $\sigma$ will always denote a measurable automorphism on $\bA$.

\begin{lemma} Let $f$ be an $L_\infty$ function measurable in $\mathcal{F}_k$. Then $f^{\sigma}(x)=f(\sigma(x))$ is also measurable in $\mathcal{F}_k$.
\end{lemma}

\begin{proof} Lemma \ref{normhom} implies that the Gowers norm $U_{k+1}$ is preserved under $\sigma$. Since the scalar product is also preserved under $\sigma$, by Lemma \ref{normchar} the proof is complete.
\end{proof}

\begin{lemma} Let $k\geq 1$-be a natural number. Then $\sigma$ permutes the rank one modules over $L_\infty(\mathcal{F}_{k-1}$. This action induces an automorphism of $\hat{\bA}_k$. By abusing the notation this action will also be denoted by $\sigma$.
\end{lemma}

Note that if $\tau$ is an arbitrary automorphism of an abelian group $A_2$ then it induces an action on $\hom^*(A_1,A_2)$.

\begin{lemma}\label{commute} The embedding $q:\hat{\bA}_k\rightarrow \hom^*(\bA,\hat{\bA}_{k-1})$ commutes with the action $\sigma$
\end{lemma}

\begin{lemma}Let $n$ be an integer co-prime to all the numbers $|A_i|,i=1,2,\dots$. Then $\sigma:x\rightarrow x^n$ is a measurable automorphism on $\bA$ and the induced action on $\hat{\bA}_k$ is again given by $x\rightarrow x^n$ for every $k\geq 1$.
\end{lemma}

\begin{proof} The statement is clear for $k=1$. Then Lemma \ref{commute} and induction on $k$ finishes the proof.
\end{proof}

We believe that the next corollary should have a simple proof but with the current path of arguments we use the full theory developed here.

\begin{corollary}[reflected character] Assume that $\phi$ is a $k$-th order character then $x\rightarrow\phi(x)\phi(-x)$ is measurable in $\mathcal{F}_{k-1}$.
\end{corollary}

\subsection{Examples}

\noindent{\it Prime characteristic:}

We have seen that if $\bA$ is of exponent $p$ (which is equivalent with saying that it is an ultra product of elementary abelian groups of exponent $p$) then all the higher order dual groups $\hat{\bA}_k$ have exponent $p$. In other words, up to abstract isomorphism, we have a full understanding of their structure.

Is this section we look at other interesting examples.
\medskip



\medskip

\noindent {\it The finite rank case:}

Assume that in the sequence $\{A_i\}_{i=1}^\infty$ each term is of rank at most $r$.
This means that they can be generated by at most $r$ elements.
Since in this case $A_i$ contains at most $n^r$ elements of order $n$ for every fixed $n$ it follows that $\bA$ has the same property and so $\hat{\bA}_1\cong\bA$ is essentially torsion free.
Now lemma \ref{estfree} and lemma \ref{dualemb} imply that $\hat{\bA}_k$ has to be torsion free for every $k\geq 2$.

\medskip

\noindent{\it High characteristic case:}

The situation of the finite rank case can be generalized. Assume that for every $n$ the number of elements of order $n$ is at most $f(n)$. Then again $\bA$ is essentially torsion free and $\hat{bA}_k$ is torsion free whenever $k\geq 2$.

An interesting phenomenon is that in the high characteristic case the divisible group $\mathbb{Q}$ can be embedded into $\bA$ and so into $\hat{\bA}_1$ showing that $\hat{\bA}_1$ is not isomorphic to the free abelian group with uncountable many generators.

\medskip



\section{The multi linear $\sigma$-algebra}

In this section we introduce the ``multi linear'' $\sigma$-algebra on the $k$-fold direct product $\bA^k$. We prove a correspondence between the $k$-th order characters in this $\sigma$-algebra and the $k$-th order characters of $\bA$.
The main motivation is the following simple observation.
Let $\phi$ be a $k$-th order pure character on $\bA$. We have that $\Delta_{t_1,t_2,\dots,t_{k+1}}\phi$ is the constant $1$ function.
This implies $\Delta_{t_1,t_2,\dots,t_k}\phi$ is a constant function the value of which is the exponential of a $k$-linear form on $\bA$. Furthermore it is easy to see that this $k$-linear form uniquely characterizes the rank one module containing $\psi$.
An advantage can be seen immediately. Let us assume that $\bA$ is of characteristic $p$. Then there are pure characters that take values which are not necessarily $p$-th root of unities. However the corresponding bilinear forms take only $p$-th roots of unities.

It will turn out in this section that this phenomenon generalizes to non pure characters.
In the general case $\Delta_{t_1,t_2,\dots,t_k}\phi(x)$ wont always be a constant function but its integral according to $x$ is not $0$ and is measurable in a $\sigma$-algebra which can be called ``multi linear'' $\sigma$-algebra. Characters of this $\sigma$-algebras are similar to multi linear functions.

\subsection{definitions and basics}

\begin{definition} Let $\bA$ and $\bB$ be two ultra product groups. Let $\mathcal{B}_1\subseteq\mathcal{A}(\bA)$ and $\mathcal{B}_2\subseteq\mathcal{A}(\bB)$ be two $\sigma$-algebras. Then their {\bf strong product} $\mathcal{B}_1\star\mathcal{B}_2$ consists of all measurable sets $S$ in $\mathcal{A}(\bA\times\bB)$ such that $(\{x\}\times \bB)\cap S$ is in $\mathcal{B}_2$ for every $x\in\bA$ and $(\bA\times\{y\})\cap S$ is in $\mathcal{B}_1$ for every fixed $y$ in $\bB$.
\end{definition}

The operation $\star$ is clearly associative.
Let $\bX_1,\bX_2,\dots,\bX_n$ be a collection of ultra product sets and let $\bX$ denote their direct product. For every subset $S\subseteq [n]$ we define the cylindric sigma algebra $\sigma_S(\bX)$ as the pre image of the ultra product $\sigma$-algebra of $\prod_{i\in S}\bX_i$ under the projection $\bX\rightarrow \prod_{i\in S}\bX_i$.
Cylindric $\sigma$-algebras were first introduced in \cite{ESz} to study hypergraph regularity and hypargraph limits. The arithmetic of these $\sigma$-algebras is completely described in \cite{ESz}.
It is easy to see that cylindric $\sigma$ algebras behave nicely under strong product in the following sense.

\begin{lemma}\label{starcyl} Let $m<n$ be a natural number, Let $\bY_1=\prod_{i=1}^m\bX_i$ and $\bY_2=\prod_{i=m+1}^n\bX_i$. Let $G_1$ be a collection of subsets of $\{1,2,\dots,m\}$ and $G_2$ be a collection of subsets of $\{m+1,m+2,\dots,n\}$. Then
$$\Bigl(\bigvee_{S\in G_1}\sigma_S(\bY_1)\Bigr)\star\Bigl(\bigvee_{S\in G_2}\sigma_S(\bY_2)\Bigr)=
\bigvee_{S_1\in G_1,S_2\in G_2}\sigma_{S_1\cup S_2}(\bX).$$
\end{lemma}

\medskip

The strong product of Fourier $\sigma$-algebras is of special interest.
For a vector $v=(v_1,v_2,\dots,v_k)$ of length $k$ with $v_i\in\mathbb{N}$ we can introduce the generalized Fourier $\sigma$-algebra $\mathcal{F}_v$ on $\bA^k$ by
$$\mathcal{F}_v=\mathcal{F}_{v_1}\star\mathcal{F}_{v_2}\star\dots\star\mathcal{F}_{v_k}.$$
If $v=(1,1,\dots,1)$ of length $k$ then we call $\mathcal{F}_v$ on $\bA^k$ the {\bf $k$-linear $\sigma$-algebra}. We will also denote the $k$-linear $\sigma$-algebra by $\star^k\mathcal{F}_1$.

\begin{lemma} The $k$-linear $\sigma$-algebra $\star^k\mathcal{F}_1$ is contained in $\mathcal{F}_k(\bA^k)$.
\end{lemma}

\begin{proof} Let us introduce $k(k+1)$ copies of $\bA$. In order to distinguish them we denote them by $\bA_{i,j}$ with $1\leq i\leq k$ and $1\leq j\leq k+1$.
Let $\bB_j=\prod_{i=1}^k\bA_{i,j}$. To prove the lemma we need to show that if a function $f$ is measurable in the $k$-linear $\sigma$ algebra then $f(x_1+x_2+\dots+x_{k+1})$ with $x_i\in\bB_i$ is measurable in
$$\bigvee_{S\subset [k+1],|S|=k}\sigma_S(\bB_1\times\bB_2\times\dots\times\bB_{k+1}).$$
For every $1\leq t\leq k$ and $1\leq s\leq k+1$ let $Q_{t,s}\subseteq [k]\times [k+1]$ denote the set $$Q_{t,s}:=\{(i,j)|i\neq t~{\rm or}~ (i=t~~{\rm and}~~ j=s)\}.$$
Lemma \ref{starcyl} implies that $f(x_1+x_2+\dots+x_{k+1})$ is measurable in each of
$$\mathcal{B}_i:=\bigvee_{j\in [k+1]}\sigma_{Q_{i,j}}(\bA^{k(k+1)}).$$
The intersection of the $\sigma$-algebras $\mathcal{B}_i$ is equal to
$$\bigvee_{f:[k]\rightarrow[k+1]}\sigma_{S_f}(\bA^{k(k+1)})$$
where $S_f$ denotes the graph of the function $f$.
This completes the proof.
\end{proof}

\subsection{The multi-linear representation map}

Motivated by the definition of the Gowers uniformity norm we introduce the functional
$\tilde{V}_k:L_\infty(\bA)\rightarrow L_\infty(\bA^k)$
defined by
$$\tilde{V}_k(f)(t_1,t_2,\dots,t_k)=\int_x\Delta_{t_1,t_2,\dots,t_k}f(x).$$
We call $\tilde{V}_k(f)$ the {\bf $k$-linear representation} of $f$.
The $k$-linear representation has two surprising properties summarized in lemma \ref{multmeas} and lemma \ref{repnonzer}.

\begin{lemma}\label{multmeas} If $f$ is any $L_\infty$ function on $\bA$ then $\tilde{V}_k(f)$ is measurable in the $k$-linear $\sigma$ algebra $\star^k\mathcal{F}_1(\bA)$.
\end{lemma}

This lemma follows immediately from the next lemma.

\begin{lemma} Let $f$ and $g$ be two $L_\infty$ functions on $\bA$. then
$h(t)=\int_x f(x)g(x+t)$ is measurable in $\mathcal{F}_1$.
\end{lemma}

\begin{proof} Let $m(x,y)=f(x)g(y)$. Let $m'$ be the projection of $m$ to the coset $\sigma$-algebra on $\bA\times\bA$ corresponding to the subgroup $\{(a,b)|a+b=0\}$.
Since $m$ is measurable in the shift invariant $\sigma$-algebra $\sigma_{\{1\}}\vee\sigma_{\{2\}}(\bA^2)$ then so is $m'$.
On the other hand $m'(x,y)=h(x-y)$. This means that $h$ must be in $\mathcal{F}_1$.
\end{proof}

\begin{corollary}[Convolution is linear] The convolution of two $L_\infty$ functions on $\bA$ is measurable in the linear $\sigma$-algebra $\mathcal{F}_1$.
\end{corollary}

\begin{lemma}Let $f\in L_\infty(\bA)$ be a function with $\|f\|_{U_{k+1}}>0$. Then for a positive measure of $t$'s $\|\Delta_tf\|_{U_k}>0$.
\end{lemma}

\begin{proof} The condition of the lemma shows that the $k$-th order Fourier expansion of $\mathbb{E}(f|\mathcal{F}_k)$ contains a nonzero term corresponding to a module $\lambda\in\hat{\bA}_k$. Let $\phi$ be a $k$-th order character from the module $\lambda$.
Then $\mathbb{E}(f\overline{\phi}|\mathcal{F}_{k-1})$ in not $0$. This means by lemma \ref{nozer} that for a positive measure of $t$'s $\mathbb{E}(\Delta_t(f\overline{\phi})|\mathcal{F}_{k-1})$ is not $0$. On the other hand
$$\mathbb{E}(\Delta_t(f\overline{\phi})|\mathcal{F}_{k-1})=\mathbb{E}(\Delta_tf|\mathcal{F}_{k-1})\Delta_t\overline{\phi}$$
This means that $\mathbb{E}(\Delta_tf|\mathcal{F}_{k-1})$ is not $0$ for a positive measure of $t$'s which completes the proof.
\end{proof}

\begin{lemma}\label{repnonzer} Let $f\in L_\infty(\bA)$ be a function with $\|f\|_{U_{k+1}}>0$. Then $\tilde{V}_k(f)$ is not the zero function.
\end{lemma}

\begin{proof} By iterating the previous lemma we obtain that for a positive measure of $k$-tuples $t_1,t_2,\dots,t_k$ the value of $\|\Delta_{t_1,t_2,\dots,t_k}f\|_{U_1}$ is bigger then $0$.
This means that for these $k$-tuples $\tilde{V}(f)$ can't be $0$.
\end{proof}

\medskip

Let $G$ denote the $k$-th dual group of $\bA^k$. Without proof we mention that elements in $G$ that can be represented by functions measurable in the multi linear $\sigma$-algebra are forming a subgroup of $G_2$ of $G$. The $k$-liner representation map induces an embedding of $\hat{\bA}_k$ into $G_2$.

\section{Finite statements}

In this part of the paper we demonstrate how to translate statements of higher order Fourier theory into finite ones. Unfortunately we loose the precise algebraic nature of many statements and instead we get approximative ones.
In general, there is an almost ``algorithmic way'' of translating infinite results from ultra product spaces to finite statements. However in some special situations one has to be careful.
An important goal of this part is to prove regularity type results for bounded functions on abelian groups. For every natural number $k$ we provide a $k$-th order regularization which is a finite version of the $k$-th order Fourier decomposition.

Regularity type results in combinatorics usually come in various strength. We have to note that both weak and strong versions have their own advantages.
Many regularity lemmas fall into a category that we call ``plain regularity''.
The general scheme for (strong versions of) plain regularity lemmas is the following.

\medskip

\noindent{\bf Rough scheme for plain regularity:}~{\it Let us fix an arbitrary function $F:\mathbb{N}\rightarrow\mathbb{R}^+$. Let $f$ be a bounded (say $\|f\|_\infty\leq 1$) function on some finite structure. Let us choose a number $\epsilon>0$. Then, if the structure is ``big enough'' then $f$ can be decomposed into three parts $f=f_s+f_e+f_r$. The function $f_s$ is the structured part and has ``complexity'' at most $n$ where $n$ is bounded in terms of $\epsilon$ and $F$. It is usually assumed that $f_s$ is bounded. The function $f_e$ is an error term with $\|f_e\|_2\leq\epsilon$. The function $f_r$ is ``quasi random'' with precision $F(n)$.}

\medskip

An example for plain regularity is the graph regularity lemma by Szemer\'edi. An example of a regularity lemma which doesn't fit into this framework is the so called hypergraph regularity lemma. Our regularity lemmas for functions on Abelian groups will also be more complicated because $k$-th order fourier analysis is closely related to the regularization of $k$-uniform hypergraphs. In this part of the paper we prove various regularity lemmas for functions. As a demonstration we highlight one of them.

\begin{definition} Let $P=\{P_1,P_2,\dots,P_n\}$ be a partition of an abelian group $A$. We say that a function $f:A\rightarrow\mathbb{C}$ is a $k$-th degree $(P,\epsilon)$-character if $$\Delta_{t_1,t_2,\dots,t_{k+1}} f(x)=g(t_1,t_2,\dots,t_{k+1},x)+r(t_1,t_2,\dots,t_{k+1},x)$$
    where $\|r\|_2<\epsilon$ and the value of $g$ depends only on the partition sets containing the sums $\sum_{i\in S}x+t_i$ where $S\subseteq [k+1]$.
\end{definition}

\begin{definition}[Complexity-I.] Let $A$ be an abelian group. We define the complexity notion $c_1(k,\underline{n},\underline{\epsilon})$ with parameters $\underline{n}=(n_1,n_2,\dots,n_{2k})\in\mathbb{N}^{2k}$ and $\underline{\epsilon}=(\epsilon_1,\epsilon_2,\dots,\epsilon_{2k})\in{\mathbb{R}^+}^{2k}$.

We say that a function $f:A\rightarrow\mathbb{C}~,~\|f\|_\infty\leq 1$ has complexity $c_1(k,\underline{n},\underline{\epsilon})$ if it has a decomposition $f=h+f_1+f_2+\dots+f_m$ with $m\leq n_{2k}~,~\|f_i\|_\infty\leq 1~,~ \|h\|_2<\epsilon_{2k}$ and $A$-has a partition $P=\{P_1,P_2,\dots,P_{2k-1}\}$ such that
\begin{enumerate}
\item each function $f_i$ is a $k$-th order $(P,\epsilon_{2k-1})$-character
\item for each partition set $P_i$ the function $1_{P_i}$ has complexity $$c_1(k-1,(n_{2k-2},n_{2k-3},\dots,n_1),(\epsilon_{2k-2},\epsilon_{2k-3},\dots,\epsilon_1)).$$
\end{enumerate}
In the case $k=0$ we say that only constant functions are of complexity $(0,n,\epsilon)$ no matter what $n$ and $\epsilon$ is.
\end{definition}

\begin{theorem}[$k$-th order Fourier-regularity-I.]\label{furreg} For every system of functions $F_i:\mathbb{N}^i\rightarrow\mathbb{R}^+~(i=0,1,2,\dots,2k-1)$ and $G:\mathbb{N}^{2k}\rightarrow\mathbb{R}^+$ there is a natural number $s$ such that every function $f:A\rightarrow\mathbb{C}~,~\|f\|_\infty\leq 1$ on an abelian group $A$ of size at least $s$ there are natural numbers $n_1,n_2,\dots,n_{2k}$ between $1$ and $s$ and a decomposition $f=f'+g$ where $\|g\|_{U_{k+1}}\leq G(n_1,n_2,\dots,n_{2k})$ and $f'$ is of complexity
$$c_1\Bigl(k,(n_1,n_2,\dots,n_{2k}),(F_{2k-1}(n_2,n_3,\dots,n_{2k}),F_{2k-2}(n_3,
n_4,\dots,n_{2k}),\dots,F_0)\Bigr).$$
\end{theorem}

\subsection{Function properties}

A property in general is just a subset of objects under investigation.
In this chapter we will talk about function properties on abelian groups.
A finite function property is a property defined on the set of all complex valued functions on finite abelian groups. An infinite function property is defined on the set of measurable functions on groups that are ultra products of finite abelian groups.

\begin{definition} Let $\mathcal{P}$ be a finite function property.
We define the infinite function property
$$cl(\mathcal{P})=\{f|f=\limo f_i~,~f_i:A_i\mapsto\mathbb{C}~,~f_i\in\mathcal{P}~,~\sup_i\|f_i\|_\infty<\infty\}$$
where $\omega$ is an arbitrary ultra filter and $\{A_i\}_{i=1}^\infty$ is an arbitrary sequence of finite abelian groups with an infinite ultra product.

\end{definition}

\begin{definition} Let $\mathcal{P}(\underline{n},\underline{\epsilon})$ be a parameterized family of function properties with parameters $\underline{n}=(n_1,n_2,\dots,n_r)\in\mathbb{N}^r$ and $\underline{\epsilon}=(\epsilon_1,\epsilon_2,\dots,\epsilon_r)\in{\mathbb{R}^+}^r$. We say that $\mathcal{P}(\underline{n},\underline{\epsilon})$ characterizes a certain infinite function property $\mathcal{P}'$ if
$$\mathcal{P}'=\bigcap_{\epsilon_r}\bigcup_{n_r}\bigcap_{\epsilon_{r-1}}\bigcup_{n_{r-1}}\dots\bigcap_{\epsilon_1}\bigcup_{n_1}cl(\mathcal{P}(\underline{n},\underline{\epsilon}))$$
\end{definition}

\subsection{Fourier regularity}

\begin{definition} Let $\mathcal{F}'_k$ denote the infinite function property that $f$ is measurable in $\mathcal{F}_k$ and $\|f\|_\infty\leq 1$.
\end{definition}

\begin{theorem}[general form of $k$-th order Fourier-regularity] Let $\mathcal{P}((n_1,n_2,\dots,n_r),(\epsilon_1,\epsilon_2,\dots,\epsilon_r))$ be a $2r$ parameter family of function properties which characterizes a property containing $\mathcal{F}'_k$. Then for every system of functions $F_i:\mathbb{N}^i\rightarrow\mathbb{R}^+~(i=0,1,2,\dots,r-1)$ and $G:\mathbb{N}^r\rightarrow\mathbb{R}^+$ there is a natural number $s$ such that for every function $f:A\rightarrow\mathbb{C}~,~\|f\|_\infty\leq 1$ on an abelian group $A$ of size at least $s$ there there are numbers $n_1,n_2,\dots,n_r$ between $1$ and $s$ and a decomposition $f=f'+g$ where $\|g\|_{U_{k+1}}\leq G(n_1,n_2,\dots,n_r)$ and $f'$ has property
\begin{equation}\label{prop}
\mathcal{P}\Bigl((n_1,n_2,\dots,n_r),(F_{r-1}(n_2,n_3,\dots,n_r),F_{r-2}(n_3,
n_4,\dots,n_r),\dots,F_0)\Bigr)
\end{equation}
\end{theorem}

\begin{proof} We go by contradiction. Assume that for a fixed set of functions there is a sequence of counterexamples $f_i:A_i\rightarrow\mathbb{C}$ for the statement with $s=i$. In particular this means that $|A_i|\geq i$ and so their ultra product $\bA$ is infinite. Let $f:\bA\rightarrow\mathbb{C}$ denote the ultra limit of $\{f_i\}_{i=1}^\infty$.
We have that $f=f'+g$ where $f'=\mathbb{E}(f|\mathcal{F}_k)$ and $\|g\|_{U_{k+1}}=0$.
The fact that $\mathcal{P}$ characterizes a property containing $\mathcal{F}_k$ implies that there is a sequence $(n_1,n_2,\dots,n_k)$ such that $f'$ is contained in
$$cl\Bigl(\mathcal{P}\Bigl((n_1,n_2,\dots,n_r),(F_{r-1}(n_2,n_3,\dots,n_r),F_{r-2}(n_3,
n_4,\dots,n_r),\dots,F_0)\Bigr)\Bigr).$$
and so there is a sequence $\{f'_i\}_{i=1}^\infty$ with ultra limit $f'$ such that
$f'$ has property (\ref{prop}).
Let $g_i=f_i-f'_i$. The ultra limit of $\{g_i\}_{i=1}^\infty$ is $g$ and so the ultra limit of $\|g_i\|_{U_{k+1}}$ is $0$. This means that there is a natural number $s\geq\max(n_1,n_2,\dots,n_r)$ with $\|g_s\|_{U_{k+1}}\leq G(n_1,n_2,\dots,n_r)$. This is a contradiction since the decomposition $f_s=f_s'+g_s$ satisfies the statement.
\end{proof}

\begin{lemma} The property characterized by the complexity notion $c_1(k,\underline{n},\underline{\epsilon})$ contains $\mathcal{F}'_k$.
\end{lemma}

\begin{proof} We go by induction on $k$. First we show that if $f\in\mathcal{F}'$ then $\forall\epsilon_{2k}~\exists n_{2k}~\forall\epsilon_{2k-1}\exists n_{2k-1}$ such that $f$ has a decomposition $f=h+f_1+f_2+\dots+f_m$ with $m\leq n_{2k}~,~\|f_i\|_\infty\leq 1~,~ \|h\|_2<\epsilon_{2k}$ and $A$-has a partition $P=\{P_1,P_2,\dots,P_{2k-1}\}$ such that
\begin{enumerate}
\item each function $f_i$ is a $k$-th order $(P,\epsilon_{2k-1})$-character
\item each partition set $P_i$ is measurable in $\mathcal{F}_{k-1}$
\end{enumerate}

To see this we consider the $k$-th order Fourier decomposition $f=f_1+f_2+\dots$. Since $f_1+f_2+\dots$ converges in $L_2$ there is a natural number $n_{2k}$ such that $\|f-f_1-f_2\dots-f_{n_{2k}}\|_2<\epsilon_{2k}$.
Lemma \ref{charsep} implies that there is a separable $\sigma$-algebra $\mathcal{B}\subset\mathcal{F}_{k-1}$ such that $f_i\in [\mathcal{B},k]^*$ for every $i$.
This means that there is a finite sub $\sigma$-algebra $\mathcal{B}_2\subset\mathcal{B}$ such that the projections $r_i$ of the functions
$h_i=\Delta_{t_1,t_2,\dots,t_{k+1}}f_i(x)$ on $\bA^{k+2}$ to the $\sigma$ algebra generated by the $\sigma$-algebras $\psi_S^{-1}(\mathcal{B}_2)$ where $S$ runs trough the subsets of $\{1,2,\dots,k+1\}$ satisfy $\|r_i-h_i\|<\epsilon_{2k-1}$ whenever $i=1,2,\dots,n_{2k}$.
It follows that the atoms of $\mathcal{B}_1$ are forming a partition $P$ such that $f_i$ is an $(P,\epsilon_{2k-1})$-character whenever $i<n_{2k}$.

Now by our induction step and by using that there are only finitely many partition sets in $P$ we obtain the formula that
$\forall\epsilon_{2k}~\exists n_{2k}~\forall\epsilon_{2k-1}\exists n_{2k-1}\dots\forall\epsilon_1~\exists n_1$ such that every partition set in $P$ is in the closure of $c_1(k-1,(n_{2k-2},n_{2k-3},\dots,n_1),(\epsilon_{2k-2},\epsilon_{2k-3},\dots,\epsilon_1))$.

By considering a finite approximation of the system $\{f_i\}_{i=1}^m$ and finite approximations of the partition sets with property $c_1(k-1,(n_{2k-2},n_{2k-3},\dots,n_1),(\epsilon_{2k-2},\epsilon_{2k-3},\dots,\epsilon_1))$
we obtain that the property characterized by $c_1$ contains $\mathcal{F}_k'$.


\end{proof}

\subsection{Finite algorithm for the $k$-th order Fourier decomposition (sketch)}

\medskip

We start with some motivation that comes from ordinary Fourier analysis.
Let $f:A\mapsto\mathbb{C}$ be a function on a finite abelian group and let

\begin{equation}\label{fegyenlo}
f=\sum_{i=1}^{|A|}\chi_i\lambda_i
\end{equation}

be its ordinary Fourier decomposition where the terms $\chi_i$ are linear characters and their coefficients are ordered in a way that $|\lambda_1|\geq|\lambda_2|\geq|\lambda_3|\geq\dots$.
Let $M_f$ be the $A\times A$ rank one matrix $ff^*$.
For all matrix operations we use the normalization introduced in chapter \ref{limop}.

The group $A$ is acting on the entries of $M_f$ in the natural way: $(x,y)^a=(x+a,y+a)$. Let $M_f'$ be the matrix obtained from $M_f$ by averaging out the action of $A$. This is: $M_f'=1/|A|\sum_{a\in A}M_f^a$. It is easy to see that
$$M_f'=\sum_{i=1}^{|A|}\chi_i\chi_i^*|\lambda_i|^2.$$
In other words the eigenvectors of $M_f$ (if they have multiplicity one !) are constant multiples of linear characters on $A$. The importance of this observation is that to define $M_f'$ we don't need the concepts of Fourier decomposition and characters but its eigenvectors (corresponding to multiplicity one eigenvalues) are linear characters.

It is natural to ask if we can obtain quadratic decomposition of $f$ in a similar way? The averaging used to obtain $M_f'$ corresponds to a projection to the trivial $\sigma$-algebra $\mathcal{F}_0$. It turns out from theorem \ref{spec} that to get a quadratic decomposition we need to use projections to $\mathcal{F}_1$. The problem is that such a projection does not have a precise finite analogy but fortunately we are able to "approximate it".

For every $\epsilon>0$ we define the (non linear) operator $P_\epsilon:A^*\rightarrow A^*$ such that
$$P_\epsilon(f)=\sum_{|\lambda_i|\geq\epsilon}\chi_i\lambda_i$$
provided that $f$ has decomposition (\ref{fegyenlo}). These operators obviously commute with the shift operators defined by $S_a(f)(x)=f(x+a)$ where $a\in A$.
Now we define $Q_{\epsilon,f}(x,y)$ to be $$P_\epsilon(a\rightarrow M_f(x+a,y+a))(0).$$
We consider $Q_{\epsilon,f}$ as a finite analogy of the operator $\mathcal{K}_2(f)$.

Very roughly speaking, if we fix a natural number $m$ and a small constant $\epsilon>0$ then if $A$ is big enough then (appropriate constant multiples of the) eigenvectors corresponding to the $m$ largest eigenvalues of $Q_{\epsilon,f}$ can be regarded as ``finite approximations'' of the quadratic decomposition of $f$.

There are however problems with this statement that can be fixed with careful work.
The biggest problem already appears in the infinite language. In corollary \ref{specor2}
we see that if $\mathcal{K}_k(f)$ has multiple eigenvalues that it require further work to read off the $k$-th order decomposition of $f$.
Unfortunately, in the finite analogy multiple eigenvalues correspond to close eigenvalues, so the above algorithm only works if the $m$ largest eigenvalues are ``well separated''. If not, then we need a refinement of the algorithm.

We also need to understand in what sense $Q_{\epsilon,f}$ is a finite version of $\mathcal{K}_2(f)$. Let $f_i:A_i\rightarrow\mathbb{C}$ be a sequence of bounded functions. We denote by $M_\epsilon$ the ultra limit of the matrices $Q_{\epsilon,f}$. It is easy to see that as $\epsilon$ goes to $0$ the matrix $M_\epsilon$ converges to $\mathcal{K}_2(f)$ in $L_2$. This will imply that the $m$ largest eigenvalues also converge in $L_2$. On the other hand theorem \ref{limop} tells us how the eigenvectors behave when the ultra limits of operators is taken.

The full algorithm will be discussed in a separate paper.

\vskip 0.2in

\noindent
Bal\'azs Szegedy
\noindent
University of Toronto, Department of Mathematics,
\noindent
St George St. 40, Toronto, ON, M5R 2E4, Canada

\end{document}